\documentclass[12pt,a4paper]{article}
\usepackage{amsmath,amsfonts,amssymb,amsthm,amscd,mathrsfs}
\usepackage{stmaryrd,yhmath}
\usepackage{cases}
\usepackage[english]{babel}
\usepackage[T1]{fontenc}
\usepackage{aeguill}
\usepackage{enumerate}
\usepackage{hyperref}

\usepackage{graphics,epsfig,color,cite,bbm}

\definecolor{light}{gray}{.9}

\newcommand{\R}[1]{\mb{R}^{#1}}

\def\R{{\mathbf R}}

\def\div{{\rm div}\,}
\def\Hess{{\rm Hess\,}}

\def\Ran{{\rm Ran}}
\def\sspan{{\rm Span}}

\def \Tr{{\rm \, Tr\,}}

\def\and {{\rm \, and \,}}

\def\Ker {{\rm \, Ker  \,}}

\makeatletter

\@addtoreset{equation}{section}
\makeatother

\newtheorem{theorem}{Theorem}[section]
\newtheorem{lemma}[theorem]{Lemma}

\newtheorem{proposition}[theorem]{Proposition}

\newtheorem{corollary}[theorem]{Corollary}

\newtheorem{theorem*}{Theorem}

\begin{document}

\title{On Witten Laplacians and Brascamp-Lieb's inequality
on manifolds with boundary}

\author{Dorian~Le~Peutrec\thanks{Laboratoire de Math\'ematiques d'Orsay, Univ. Paris-Sud, CNRS, Universit\'e
Paris-Saclay, 91405 Orsay, France (dorian.lepeutrec@math.u-psud.fr)}
}

\maketitle

\begin{abstract}
In this paper, we  derive from the supersymmetry
of the Witten Laplacian
Brascamp-Lieb's type inequalities for general differential
forms  on  compact Riemannian  manifolds with  boundary.
In addition to the supersymmetry, our results essentially follow
from suitable decompositions of  the quadratic forms associated 
with the Neumann and Dirichlet self-adjoint realizations of the   Witten Laplacian.
They moreover imply the usual Brascamp-Lieb's inequality and its generalization
to compact Riemannian manifolds without boundary. 
\end{abstract}

\textit{Brascamp-Lieb's inequality,  Witten Laplacian, Riemannian manifolds with boundary, Supersymmetry, Bakry-\'Emery tensor.}

\section{Introduction}

\subsection{Context and aim of the paper}

\noindent Let $V\in \mathcal C^{2}(\mathbf R^{n},\mathbf R)$
be a strictly convex function such that $e^{-V}\in L^{1}(\mathbf R^{n})$
and let $\nu$ be the probability measure defined by $d\nu:=\frac{e^{-V}}{\int_{\R^{n}}e^{-V}dx}\,dx$.
The classical Brascamp-Lieb's inequality proven in~\cite{BrLi} states that 
 every smooth   compactly supported function  $\omega$
satisfies the estimate
\begin{equation}
  \label{eq.BL0}
\int_{\mathbf R^{n}} \big|\,\omega\,-\,\big(\int_{\mathbf R^{n}}\omega~d\nu\big)\,\big|^{2}~d\nu
  \  \leq\  
\int_{\mathbf R^{n}} \big(\Hess V\big)^{-1}(\nabla \omega,\nabla \omega)~d\nu\,.
\end{equation}
This inequality and suitable variants have since been  e.g. used in works
such as \cite{HeSj,Sjo,NaSp,Hel,BaJeSj,BaMo,BaMo2}
studying correlation  asymptotics   in statistical mechanics. 
The latter works exploit in particular crucially some relations of the following type
and which at least go back to the work of Helffer and Sj\"ostrand \cite{HeSj}:
\begin{equation}
\label{eq.HS0}
\big\| \,\eta - \langle\eta,\frac{e^{-\frac V2}}{\| e^{-\frac V2} \|} \rangle\frac{e^{-\frac V2}}{\| e^{-\frac V2} \|}\, \big\|^{2}
\ =\ \langle \,(\Delta_{\frac V2}^{(1)})^{-1}\, \big(\,d_{\frac V2}\,\eta\,\big)\,,\,
d_{\frac V2}\,\eta\, \rangle,
\end{equation}
where $\eta \in \mathcal C^{\infty}_{\text{c}}(\mathbf R^{n})$, $\langle\cdot,\cdot \rangle$  and $\|\cdot\|$
stand for the usual $L^{2}(dx)$ inner product and norm,
$
d_{\frac V2}:= d+d\frac{V}{2}
$
and $\Delta_{\frac V2}^{(1)}$ is the Witten Laplacian acting on $1$-forms
(or equivalently on vector fields) which is given by
\begin{equation}
\label{eq.Witten-1-forms}
\Delta_{\frac V2}^{(1)}\ :=\ \Delta_{\frac V2}^{(0)}\otimes \text{Id}\,+\,\Hess V\ =\ 
\big(-\Delta+|\nabla \frac V2|^{2}-\Delta\frac V2\big)\otimes \text{Id}\,+\,\Hess V\,.
\end{equation}
 In the last relation, 
 \begin{equation}
\label{eq.Witten-0-forms}
 \Delta_{\frac V2}^{(0)}\ :=\ -\Delta+|\nabla \frac V2|^{2}-\Delta\frac V2\ =\ \big(-\div+\nabla \frac{V}{2}\big)\big(\nabla+\nabla \frac{V}{2}\big)\ =\ d^{*}_{\frac V2}\,d_{\frac V2}
 \end{equation}
 denotes the Witten Laplacian acting on functions (or equivalently on $0$-forms).
The Witten Laplacian, initially introduced in \cite{Wit}, is more generally defined
on the full algebra of differential forms  and is nonnegative and essentially self-adjoint
(when acting on smooth compactly supported forms)
 on the space of $L^{2}(dx)$
differential forms.
It is moreover
 supersymmetric, which essentially amounts, when
restricting our attention to the interplay between  $\Delta_{\frac V2}^{(0)}$ and $\Delta_{\frac V2}^{(1)}$,
to
the
intertwining relation
$$
\forall\,\eta\,\in\,\mathcal C^{\infty}_{\text{c}}(\mathbf R^{n})\, ,\quad
d_{\frac V2}\, \Delta_{\frac V2}^{(0)}\,\eta\ =\ \Delta_{\frac V2}^{(1)}\,d_{\frac V2}\,\eta\,,
$$
 which enables to prove relations of the type \eqref{eq.HS0} (when $\Delta_{\frac V2}^{(1)}$ is invertible).
 The nonnegativity of $\Delta_{\frac V2}^{(0)}$ together with the relations 
 \eqref{eq.HS0} and \eqref{eq.Witten-1-forms} then 
 easily leads  to \eqref{eq.BL0} when $V$ is strictly convex (at least formally) taking finally $\omega:=e^{\frac V 2}\eta$. 
 To connect  to some spectral properties of $\Delta_{\frac V2}^{(0)}$,
 the relation \eqref{eq.HS0}
 together with the lower bound
 $\Delta_{\frac V2}^{(1)}\geq c$ for some $c>0$ -- which is in particular satisfied if
  $\Hess V\geq c$ --
 implies, according to formula  \eqref{eq.Witten-0-forms}, 
  a spectral gap greater or equal to $c$
for  $\Delta_{\frac V2}^{(0)}$ 
(its kernel  being $\sspan\{e^{-\frac V2}\}$ as it can be seen from \eqref{eq.Witten-0-forms}).
In addition to the already mentioned \cite{Sjo,Hel} making extra 
assumptions on $V$, we refer especially 
to the very complete  \cite{Joh}
  for precise statements and proofs in relation with the above discussion.\\
  
\noindent More generally, in the case of a Riemannian manifold without boundary $\Omega$,
it is also well known that an inequality of the type \eqref{eq.BL0} holds if one
replaces
 $\Hess V$ (and the condition $\Hess V>0$ everywhere)
 by
the following quadratic form, sometimes called the Bakry-\'Emery\,(-Ricci) tensor,  
$${\rm Ric}\,+\, \Hess V\quad\text{(and if we  assume its strict positivity everywhere)}\,,$$ 
${\rm Ric}$ denoting the Ricci tensor. We refer for example to \cite[Theorem~4.9.3]{BaGeLe}
for a precise statement whose proof relies on the
supersymmetry  of the counterpart of the  Witten Laplacian in the weighted space
$L^{2}(\Omega,e^{-V}d\,\text{Vol}_{\Omega})$, sometimes called the weighted Laplacian and more precisely defined when acting on functions by
$$
L^{(0)}_{V}\ :=\ e^{\frac V2}\big(-\Delta+|\nabla \frac V2|^{2}-\Delta\frac V2\big)e^{-\frac V2}\ =\ -\Delta+\nabla V\cdot\nabla\,.
$$
This operator, unitarily equivalent to $\Delta_{\frac V2}^{(0)}$,
 is an important model of the Bakry-\'Emery theory of diffusion processes and
we refer especially in this direction to the pioneering work of Bakry and \'Emery \cite{BaEm} or to the book \cite{BaGeLe} for an overview
of the concerned literature. On its side, the Bakry-\'Emery tensor 
${\rm Ric}+ \Hess V$ -- named after \cite{BaEm} but first introduced by Lichnerowicz in \cite{Lic} -- is the natural counterpart
   of the Ricci tensor  ${\rm Ric}$ in the weighted Riemannian manifold $(\Omega,e^{-V}d\,\text{Vol}_{\Omega})$
   and we  refer for example to \cite{Lic,Lot} for some of its geometric properties.
   Let us also mention e.g. \cite{LoVi} extending this notion to metric measure
spaces.\\

\noindent
In this paper, we derive from the supersymmetry
of the Witten Laplacian
Brascamp-Lieb's type inequalities for general differential
forms  on a Riemannian  manifold with a boundary.
In addition to the supersymmetry, our results essentially follow
from suitable decompositions of  the quadratic forms associated 
with the self-adjoint Neumann and Dirichlet realizations of the Witten Laplacian
stated in Theorem~\ref{th.Witten-Bochner}.
When restricting to the interplay between $0$- and $1$-forms,
they
imply in particular the already mentioned results in the case
of $\mathbf R^{n}$ or of a compact manifold with empty boundary
as well as some results recently obtained by Kolesnikov and  Milman
in \cite{KoMi} in the case of a compact manifold with a boundary  
(see indeed Corollaries~\ref{co.BL}~and~\ref{co.main} and the corresponding remarks).

\subsection{Decomposition formulas}

\noindent
Let $\left(\Omega,g=\langle \cdot,\cdot\rangle\right)$ be   a smooth $n$-dimensional oriented connected 
and compact Riemannian manifold 
with boundary
$\partial\Omega$.
The cotangent (resp. tangent) bundle of $\Omega$ is denoted by $T^{*}\Omega$
(resp. $T\Omega$) and the exterior fiber
bundle by $\Lambda T^{*}\Omega=\oplus_{p=0}^{n}\Lambda^{p}T^{*}\Omega$.
 The fiber
bundles $T^{*}\partial\Omega$,
$T\partial\Omega$,
and $\Lambda T^{*}\partial \Omega=\oplus_{p=0}^{n-1}\Lambda^{p}T^{*}\partial\Omega$
 are defined similarly.
 The (bundle) scalar product
on $\Lambda^{p}T^{*}\Omega$  
inherited from $g$ is denoted by $\langle \cdot,\cdot \rangle_{\Lambda^{p} }$.
The space of $\mathcal{C}^{\infty}$,
$L^{2}$, etc.  sections of any of the above fiber bundles $E$, over
 $O=\Omega$ or $O=\partial \Omega$, are respectively denoted
 by $\mathcal{C}^{\infty}(O,E)$,
 $L^{2}(O,E)$, etc.. The more compact notation $\Lambda^{p}\mathcal{C}^{\infty}$, $\Lambda^{p}L^{2}$,
 etc. will also be used instead of $\mathcal{C}^{\infty}(\Omega,\Lambda^{p}T^{*}\Omega)$,
 $L^{2}(\Omega,\Lambda^{p}T^{*}\Omega)$, etc.
 and we will denote by $\mathcal L(\Lambda^{p}T^{*}\Omega)$ the space of smooth
  bundle endomorphisms of  $\Lambda^{p}T^{*}\Omega$.
  The $L^{2}$ spaces are those associated
with the respective unit volume forms $\mu$ and $\mu_{\partial\Omega}$ for the Riemannian structures on
$\Omega$ and  on $\partial \Omega$. The $\Lambda^{p}L^{2}$ scalar product and norm corresponding to $\mu$ will be denoted
by $\langle \cdot,\cdot\rangle_{\Lambda^{p}L^{2}}$ and $\|\cdot\|_{\Lambda^{p}L^{2}}$ or more simply by  $\langle \cdot,\cdot\rangle_{L^{2}}$ and $\|\cdot\|_{L^{2}}$
when no confusion is possible.\\

\noindent 
We denote by $d$ the exterior differential on $\mathcal{C}^{\infty}(\Omega,\Lambda T^{*}\Omega)$
and by $d^{*}$ its formal adjoint with respect to the $L^{2}$ scalar product.
The Hodge Laplacian is then defined on $\mathcal{C}^{\infty}(\Omega, \Lambda
T^{*}\Omega)$ by
\begin{equation}
\Delta\  := \ \Delta_{H}\ :=\ d^{*}d+dd^{*}\ =\ (d+d^{*})^{2}\,.
\end{equation}

\noindent For a (real) smooth function $f$, the distorted differential operators $d_{f}$ and $d^{*}_{f}$ are defined on 
$\mathcal{C}^{\infty}(\Omega, \Lambda
T^{*}\Omega)$ by
\begin{equation}
\label{eq.d-dstar-f-h}
d_{f}\ :=\ e^{-f}\,d\, e^{f}\quad\text{and}\quad d_{f}^{*}
\ :=\ e^{f}\, d^{*}\,  e^{-f}\,,
\end{equation}
and the Witten Laplacian $\Delta_{f}$ is 
defined on
$\mathcal{C}^{\infty}(\Omega, \Lambda
T^{*}\Omega)$
 similarly as the Hodge Laplacian
by
\begin{equation}
\Delta_{f}\ :=\ d_{f}^{*}d_{f}+d_{f}d_{f}^{*}\ =\ (d_{f}+d_{f}^{*})^{2}\,.
\end{equation}
Note
the supersymmetry structure of the Witten Laplacian acting on the complex of
differential forms: 
 for every $u$ in
 $\mathcal{C}^{\infty}(\Omega,\Lambda^{p}T^{*}\Omega)$, it holds
\begin{align}
  \label{eq.commut}
\Delta_{f}^{(p+1)}d_{f}^{(p)}u\ =\ d_{f}^{(p)}\Delta_{f}^{(p)}u\ \ 
\text{and}\ \ 
 \Delta_{f}^{(p-1)}d_{f}^{(p-1),*}u\ =\ d_{f}^{(p-1),*}\Delta_{f}^{(p)}u\,.
\end{align}
The Witten Laplacian $\Delta_{f}^{(p)}$
(the superscript ${(p)}$ means that we are considering its action on differential
$p$-forms)
 extends in the distributional sense into an operator acting on the Sobolev space
$\Lambda^{p}H^{2}$ and is nonnegative and self-adjoint on the flat space  $\Lambda^{p}L^{2}=\Lambda^{p}L^{2}(d\mu)$
once endowed with appropriate Dirichlet or Neumann type boundary conditions
(see indeed \cite{HelNi1,Lep} and Section~\ref{se.Wittenavecbord}).
These self-adjoint extensions are respectively denoted
by  $\Delta_{f}^{\mathbf t,(p)}$ and $\Delta_{f}^{\mathbf n,(p)}$, their respective domains being
given by
\begin{equation}
\label{eq.domain-Delta-t}
D(\Delta_{f}^{\mathbf t,(p)})\ =\  \left\{\omega\in \Lambda^{p}H^{2}\,,\ 
  \mathbf{t}\omega=0\quad \text{and}\quad   \mathbf{t}d^{*}_{f}\omega=0\quad\text{on}\quad\partial\Omega \right\}
\end{equation}
and
\begin{equation}
\label{eq.domain-Delta-n}
D(\Delta_{f}^{\mathbf n,(p)})\ =\  \left\{\omega\in \Lambda^{p}H^{2}\,,\ 
  \mathbf{n}\omega=0\quad\text{and}\quad  \mathbf{n}d_{f}\omega=0\quad\text{on}\quad\partial\Omega \right\}\,.
\end{equation}
In the above two formulas, $\mathbf n\eta$ and $\mathbf t\eta$ stand
respectively
 for the normal and tangential components of the form $\eta$, see
 \eqref{eq.def-tan} and \eqref{eq.def-nor} in the following section  for a precise
definition. 
For $\mathbf b\in\{\mathbf t,\mathbf n\}$,
the quadratic form 
associated with
$\Delta_{f}^{\mathbf b,(p)}$ 
is denoted by $\mathcal D_{f}^{\mathbf b,(p)}$.
Its domain is given by
\begin{equation}
\label{eq.domain-Delta-Q}
\Lambda^{p}H^{1}_{\mathbf b} \ :=\  
\left\{\omega\in \Lambda^{p}H^{1}\,,\ 
  \mathbf{b}\omega=0\quad\text{on}\quad\partial\Omega \right\}
\,,
\end{equation}
and we have,  for every $\omega\in\Lambda^{p}H^{1}_{\mathbf b}$,
\begin{equation}
\label{eq.Dfh}
\mathcal D_{f}^{\mathbf b,(p)}(\omega) \ :=\  
\mathcal D_{f}^{\mathbf b,(p)}(\omega,\omega)
\  =\  \langle d_{f}\omega,d_{f}\omega\rangle_{ L^{2}}+
\langle d^{*}_{f}\omega,d^{*}_{f}\omega\rangle_{ L^{2}}.
\end{equation}
More details about these self-adjoint realizations are given  in Section~\ref{se.Wittenavecbord}.\\

\noindent The  different Brascamp-Lieb's type inequalities stated
in this work arise  from the   
following integration by parts formulas
relating the quadratic forms $\mathcal D_{f}^{\mathbf t,(p)}$ and $\mathcal D_{f}^{\mathbf n,(p)}$
with the geometry of $\Omega$.
In order to lighten this presentation, 
some notations involved in these formulas will only be precisely defined in the next section:
$\partial_{n}f$ denotes the normal derivative of $f$ along the boundary (see \eqref{eq.def-normal-der}), 
${\rm{Ric}}^{(p)}$ and $ \Hess^{\!(p)}\!f $ respectively denote   the smooth
 bundle symmetric endormorphism
of $\Lambda^{p}T^{*}\Omega$ 
defined from the Weitzenb\"ock formula in
 \eqref{eq.Ric-p} and the one
 canonically 
associated with $\Hess f$ (see \eqref{eq.Hess-p}),
 and, for $\mathbf b \in \{\mathbf n,\mathbf t\}$,  
$\mathcal K_{\mathbf b}^{(p)} \in \mathcal L(\Lambda^{p}T^{*}\Omega\big|_{\partial\Omega})$
is defined by means of the second fundamental form of $\partial\Omega$ in
\eqref{eq.Kn1}--\eqref{eq.Ktp}.

\begin{theorem}
\label{th.Witten-Bochner}
Let $\omega\in \Lambda^{p}H_{\mathbf{b}}^{1}$
 with $\mathbf b \in \{\mathbf n,\mathbf t\}$ and $p\in\{0,\dots,n\}$.
It holds
\begin{multline}
  \label{eq.Witten-Bochner}
\mathcal D^{\mathbf b,(p)}_{f}(\omega)\ =\ 
 \| e^{f }\omega\|^{2}_{\dot H^{1}(e^{-2f}d\mu)}
+\langle \left({\rm{Ric}}^{(p)}+2\,\Hess^{\!(p)}\!f\right)\omega ,\omega \rangle_{L^{2}}\\
+ \int_{\partial\Omega} \langle \mathcal K_{\mathbf b}^{(p)}\omega ,\omega \rangle_{\Lambda^{p}}~d\mu_{\partial\Omega}
-2\,\mathbf1_{\mathbf t}(\mathbf b)\,\int_{\partial \Omega}\langle\omega,\omega
\rangle_{\Lambda^{p}}\,\partial_{n}f~d \mu_{\partial \Omega}\,,
\end{multline}
where $\mathbf1_{\mathbf t}(\mathbf b)=1$ if $\mathbf b=\mathbf t$ and $0$ if $\mathbf b=\mathbf n$,
and 
$$\| \cdot\|^{2}_{\dot H^{1}(e^{-2f}d\mu)}\ :=\ 
\| \cdot\|^{2}_{H^{1}(e^{-2f}d\mu)}-\| \cdot\|^{2}_{L^{2}(e^{-2f}d\mu)}\,.$$
\end{theorem}

\ \ 

\noindent When $f=0$, we recover
Theorems~2.1.5 and~2.1.7 of \cite{Sch}
which were generalizing  results in the boundaryless case 
due to Bochner 
for $p=1$ and to Gallot and Meyer for general $p$'s (see \cite{Boc,GaMe}).
These results allow in particular to draw topological 
conclusions on the cohomology of $\Omega$  from its geometry. 
When the boundary $\partial\Omega$ is not empty,  
 the relative and absolute
cohomologies of $\Omega$ (corresponding respectively
to the Dirichlet and Neumann boundary conditions) have to be considered
(see \cite[Section~2.6]{Sch}). To be more precise,
note from Theorem~\ref{th.Witten-Bochner} that
for any $p\in\{0,\dots,n\}$, 
the (everywhere) positivity of the quadratic form ${\rm{Ric}}^{(p)}+2\,\Hess^{\!(p)}\!f$
together with the nonnegativity of $\mathcal K_{\mathbf n}^{(p)}$
(resp. of $\mathcal K_{\mathbf t}^{(p)}-2\,\partial_{n}f$)
implies the lower bounds (in the sense of quadratic forms)
$$\Delta_{f}^{\mathbf b,(p)}\ \geq\  {\rm{Ric}}^{(p)}+2\,\Hess^{\!(p)}\!f\ >\ 0\quad\text{(\,$\mathbf b\in\{\mathbf t,\mathbf n\}$\,)}$$
 for  the Witten Laplacian
and hence the triviality of its kernel
which is isomorphic to the $p$-th absolute (resp. relative) cohomology group of $\Omega$
when  $f=0$.

\subsection{Consequences: Brascamp-Lieb's type inequalities}

\noindent We now define $V:=2f$, 
the 
probability measure $\nu$ associated with $V$ by
$$d\nu\ :=\ 
\frac{e^{-V }}{\int_{\Omega} e^{-V }d\mu}d\mu
\ =\ \frac{e^{-2f}}{\|e^{-f }\|^{2}_{L^{2}}}d\mu\,,
$$
and the weighted Laplacian  acting on $p$-forms $L^{(p)}_{V}$ by
\begin{equation}
\label{eq.LVh-p}
L^{(p)}_{V}\ :=\ e^{f}\,\Delta^{(p)}_{f}\,e^{-f}\,.
\end{equation}
The latter operator acting on the weighted space
$\Lambda^{p }L^{2}(e^{-V}d\mu)$ is then unitarily equivalent to $\Delta^{(p)}_{f}$
(acting on the flat space)
and
we denote by $L^{\mathbf t,(p)}_{V}$ and $L^{\mathbf n,(p)}_{V}$ the nonnegative
self-adjoint unbounded operators on $\Lambda^{p }L^{2}(e^{-V}d\mu)$
associated with $\Delta_{f}^{\mathbf t,(p)}$ and $\Delta_{f}^{\mathbf n,(p)}$
via \eqref{eq.LVh-p}. Their respective domains are easily deduced from 
\eqref{eq.domain-Delta-t}, \eqref{eq.domain-Delta-n}, and \eqref{eq.LVh-p}.\\

\noindent
We denote moreover, for $p\in\{0,\dots,n\}$, by
 $\Lambda^{p}L^{2}(d\nu)$, $\Lambda^{p}H^{1}(d\nu)$,
 $\langle\cdot,\cdot \rangle_{L^{2}(d\nu)}$ and $\|\cdot\|_{L^{2}(d\nu)}$  the weighted
 Lebesgue and Sobolev spaces, $L^{2}$ scalar product and $L^{2}$ norm. 
 We also denote by $\Lambda^{p}H_{\mathbf b}^{1}(d\nu)$
 the set of the $\omega\in \Lambda^{p}H^{1}(d\nu)$ such that $\mathbf b \omega=0$
 on $\partial\Omega$, which is the domain of the quadratic form 
 associated with $L^{\mathbf b,(p)}_{V}$
 according to \eqref{eq.domain-Delta-Q} and $\eqref{eq.LVh-p}$. Since we are working on a compact
 manifold, note  that
 $\Lambda^{p}H_{\mathbf b}^{1}(d\nu)$ is  nothing
 but $\Lambda^{p}H_{\mathbf b}^{1}$ (algebraically and topologically).\\

 \noindent
Playing  with the supersymmetry,
  we easily get from Theorem~\ref{th.Witten-Bochner} the following  
Brascamp-Lieb's type inequalities for differential forms,
where for any $\mathbf b\in\{\mathbf n,\mathbf t\}$ and $p\in\{0,\dots,n\}$,
 $\pi_{\mathbf b}=\pi^{(p)}_{\mathbf b}$ denotes
the orthogonal projection on $\Ker(L^{\mathbf b,(p)}_{V})$.

\begin{theorem}[Brascamp-Lieb's inequalities for differential forms]
\label{th.main}~\
\begin{enumerate}
\item Let $p\in\{0,\dots,n\}$ and let us assume that $\mathcal K_{\textbf n}^{(p)}\geq 0$ everywhere
on $\partial\Omega$ and that
${\rm{Ric}}^{(p)}_{V}:={\rm{Ric}}^{(p)}+\Hess^{\!(p)} V>0$ everywhere on $\Omega$
 (in the sense of quadratic forms). 
It then holds:
\begin{itemize}
\item[i)] if $p>0$, we have for every $\omega\in \Lambda^{p-1}H_{\mathbf n}^{1}(d\nu)$
such that $d^{*}_{V}\omega =0$:
\begin{equation*}
\left\| \omega - \pi_{\mathbf n} \omega \right\|_{L^{2}(d\nu)}^{2} \ \leq\ 
\int_{\Omega} \big\langle \big({\rm{Ric}}^{(p)}_{V}\big)^{-1}d\omega\,,\,d \omega\big\rangle_{\Lambda^{p}}~d\nu\,,
\end{equation*}
\item[ii)] if $p< n$, we have for every $\omega\in \Lambda^{p+1}H_{\mathbf n}^{1}(d\nu)$
such that $d\omega=0$:
\begin{equation*}
\left\| \omega - \pi_{\mathbf n} \omega \right\|_{L^{2}(d\nu)}^{2} \ \leq\ 
\int_{\Omega} \big\langle \big({\rm{Ric}}^{(p)}_{V}\big)^{-1}d_{V}^{*}\omega\,,\,d_{V}^{*} \omega\big\rangle_{\Lambda^{p}}~d\nu\,.
\end{equation*}
\end{itemize}
\item Assume similarly that $\mathcal K_{\textbf t}^{(p)}-\partial_{n}V\geq 0$ everywhere on $\partial\Omega$ and that
${\rm{Ric}}^{(p)}_{V}>0$ everywhere on $\Omega$. It then holds:
\begin{itemize}
\item[i)] if $p>0$, we have for every $\omega\in \Lambda^{p-1}H_{\mathbf t}^{1}(d\nu)$
such that $d^{*}_{V}\omega =0$:
\begin{equation*}
\left\| \omega - \pi_{\mathbf t} \omega \right\|_{L^{2}(d\nu)}^{2} \ \leq\ 
\int_{\Omega} \big\langle \big({\rm{Ric}}^{(p)}_{V}\big)^{-1}d\omega\,,\,d \omega\big\rangle_{\Lambda^{p}}~d\nu\,,
\end{equation*}
\item[ii)]  if $p<n$, we have for every $\omega\in \Lambda^{p+1}H_{\mathbf t}^{1}(d\nu)$
such that $d\omega=0$:
\begin{equation*}
\left\| \omega - \pi_{\mathbf t} \omega \right\|_{L^{2}(d\nu)}^{2} \ \leq\ 
\int_{\Omega} \big\langle \big({\rm{Ric}}^{(p)}_{V}\big)^{-1}d_{V}^{*}\omega\,,\,d_{V}^{*} \omega\big\rangle_{\Lambda^{p}}~d\nu\,.
\end{equation*}
\end{itemize}
\end{enumerate}
\end{theorem}

\noindent
In the case $p=1$, the points 1.i) and 2.i) of Theorem~\ref{th.main}
take a simpler form. Every $\omega\in \Lambda^{0}H^{1}(d\nu)$
satisfies indeed $d^{*}_{V}\omega=0$. Moreover, 
we have simply 
$$\Lambda^{0}H^{1}_{\mathbf n}(d\nu)\ =\ H^{1}(d\nu)\quad\text{and}\quad
\Ker(L^{\mathbf n,(p)}_{V})\ =\  \sspan\{1\}$$
as well as
 $$\Lambda^{0}H^{1}_{\mathbf t}(d\nu)=H_{0}^{1}(d\nu)
 \quad\text{and}\quad
\Ker(L^{\mathbf t,(p)}_{V})\ =\ \{0\}\,.$$
Defining the mean of $u\in L^{2}(d\nu)$
by
 $\langle u\rangle_{\nu}:=\langle u, 1\rangle_{L^{2}(d\nu)}$, 
 we then immediately get from Theorem~\ref{th.main}
 (together with \eqref{eq.Kn1} and \eqref{eq.Kt1})
 the following (where $\mathcal K_{1}$ denotes the shape operator defined
 in the next section, in \eqref{eq.shape})

\begin{corollary}
\label{co.BL}
\begin{itemize}
\item[i)] Assume that the shape operator $\mathcal K_{1}$ is nonpositive everywhere on $\partial\Omega$
and that ${\rm{Ric}}+\Hess V>0$ everywhere on $\Omega$. It then holds: for every  $\omega\in H^{1}(d\nu)$,
\begin{equation}
  \label{eq.BL1}
\left\| \omega - \langle \omega\rangle_{\nu} \right\|_{L^{2}(d\nu)}^{2} \ \leq\ 
\int_{\Omega} \big({\rm{Ric}}+\Hess V\big)^{-1}(\nabla \omega,\nabla \omega)~d\nu\,.
\end{equation}
\item[ii)] Assume similarly
 that  $-\Tr(\mathcal K_{1})-\partial_{n}V\geq 0$ everywhere
on $\partial\Omega$ and that ${\rm{Ric}}+\Hess V>0$ everywhere on $\Omega$. It then holds:
 for every  $\omega\in H_{0}^{1}(d\nu)$,
\begin{equation}
  \label{eq.BL2}
\left\| \omega \right\|_{L^{2}(d\nu)}^{2}\ \leq\  
\int_{\Omega} \big({\rm{Ric}}+\Hess V\big)^{-1}(\nabla \omega,\nabla \omega)~d\nu\,.
\end{equation}
\end{itemize}
\end{corollary}

\noindent When $\Omega\setminus\partial\Omega$ appears to be a smooth open subset
of $\mathbf R^{n}$, ${\rm{Ric}}$ and ${\rm{Ric}}^{(p)}$  vanish and the latter corollary as well as 
Theorem~\ref{th.main} then write in a simpler way just
relying on a  control from below of $\Hess V$ or $\Hess^{\!(p)} V$
instead of ${\rm{Ric}}^{(p)}_{V}={\rm{Ric}}^{(p)}+\Hess^{\!(p)} V$.
One recovers in particular the usual Brascamp-Lieb's inequality when 
$\Omega=\mathbf R^{n}$: even if  $\Omega$
has been assumed compact here, we recover 
the estimate \eqref{eq.BL0} 
for a probability measure $d\nu$ on $\mathbf R^{n}$ using  the first  point of Corollary~\ref{co.BL} for the family of  measures
$\left(\frac{1}{\nu(B(0,N))}d\nu\big|_{B(0,N)}\right)_{N\in\mathbf N}$ and letting $N\to+\infty$ since $B(0,N)$ is convex; see also \cite{Joh}.\\

\noindent The above results can be useful for semiclassical
problems involving the low spectrum of semiclassical Witten Laplacians  (or equivalently of semiclassical weighted Laplacians)
in large dimension, such as problems dealing with correlation asymptotics,  under some  suitable (and uniform in the dimension)  estimates on the eigenvalues of $\Hess V$ (and then of $\Hess^{\!(p)} V$) on some parts of $\Omega$. We refer for example to \cite{HeSj,BaJeSj,BaMo,BaMo2} or to the more recent \cite{DiLe}   for some works
exploiting this kind of estimates. Let us recall that 
we consider in this setting, for a small parameter $h>0$, $\frac fh$ and $\frac V h$ instead of $f$ and $V$,
and $h^{2}\Delta_{\frac f h}^{(p)}$ instead of $\Delta_{f }^{(p)}$
for the usual semiclassical Schr\"odinger operator form. 
Note then
from ${\rm{Ric}}^{(p)}_{\frac Vh}=h^{-1}(h\,{\rm{Ric}}^{(p)}+\Hess^{\!(p)} V)$
 that the curvature effects due to ${\rm{Ric}}^{(p)}$ become negligible
at the semiclassical limit $h\to 0^{+}$ under the condition $\Hess^{\!(p)} V>0$
everywhere on $\Omega$. 
To apply  Theorem~\ref{th.main} for any small $h>0$ in the Neumann case
under this condition then only requires 
the additional $h$-independent condition
$\mathcal K_{\textbf n}^{(p)}\geq 0$
 everywhere on $\partial\Omega$.
In the Dirichlet case, the required additional condition  becomes
$h\,\mathcal K_{\textbf t}^{(p)} -\partial_{n}V\geq 0$,
 which requires in particular
 $\partial_{n}V\leq 0$  everywhere on $\partial\Omega$.
The point ii) of Corollary~\ref{co.BL} is thus irrelevant in this  case.\\

\noindent
Let us lastly underline that to prove
Theorem~\ref{th.main} (and then Corollary~\ref{co.BL}),
we only use the supersymmetry structure and the  relation
$$
\Delta_{f}^{\mathbf b,(p)}\ \geq\ {\rm{Ric}}^{(p)}+2\,\Hess^{\!(p)} f>0
$$
implied by Theorem~\ref{th.Witten-Bochner}
 together with the hypotheses of  Theorem~\ref{th.main}.
However,
a control from below of the restriction $
\Delta_{f}^{\mathbf b,(p)}\big|_{\Ran\,d_{f}}
$ for the points 1.i) and 2.i) (resp. of $
\Delta_{f}^{\mathbf b,(p)}\big|_{\Ran\,d^{*}_{f}}
$ for the points 1.ii) and 2.ii))
 would actually be sufficient as it can be seen
 by looking for example  at the further relation
 \eqref{eq.HS} generalizing \eqref{eq.HS0} (see also
 Proposition~\ref{pr.ortho-decomp}
for more details about the latter restrictions). 
  The 
specific form of the nonnegative first term in the r.h.s. of  
  the
integration by parts formula \eqref{eq.Witten-Bochner}
stated  in  Theorem~\ref{th.Witten-Bochner} is moreover not used,
i.e. only its nonnegativity comes into play.
 When
$p=1$, we can easily slightly improve Corollary~\ref{co.BL}
taking advantage of this nonnegative term which allows to compare 
$\Delta_{f}^{\mathbf b,(1)}\big|_{\Ran\,d_{f}}$
(or equivalently $L_{V,h}^{\mathbf b,(1)}\big|_{\Ran\, d}$)
with the so-called $N$-dimensional Bakry-\'Emery tensor
\begin{equation}
\label{eq.BakryEmeryCurvature}
{\rm{Ric}}_{V,N}\ :=\ {\rm Ric}\,+\,\Hess V\,-\,\frac{1}{N-n}\,dV\otimes dV\,,
\end{equation}
where $N\in (-\infty,+\infty]$ and, when $N=n$, ${\rm{Ric}}_{V,n}$ is defined iff $V$
is constant. The hypotheses of Corollary~\ref{co.BL}
require in particular the (everywhere) positivity of ${\rm{Ric}}_{V,+\infty}$
and we have more generally the

\begin{corollary}
\label{co.main} In the following, 
we assume that $N\in(-\infty,0]\cup[n,+\infty]$.
\begin{itemize}
\item[i)]  Assume that $\mathcal K_1\leq 0$ everywhere on $\partial\Omega$
and that ${\rm{Ric}}_{V,N}>0$ everywhere on $\Omega$.
It then holds:
for every  $\omega\in H^{1}(d\nu)$,
\begin{equation*}
\left\| \omega - \langle \omega\rangle_{\nu} \right\|_{L^{2}(d\nu)}^{2} \ \leq\ \frac{N-1}{N}\,
\int_{\Omega} \big({\rm{Ric}}_{V,N}\big)^{-1}(\nabla \omega,\nabla \omega)~d\nu\,.
\end{equation*}
\item[ii)] Assume similarly  that $-\Tr(\mathcal K_{1})-\partial_{n}V\geq 0$
everywhere on $\partial\Omega$ and that ${\rm{Ric}}_{V,N}>0$ on $\Omega$. It then holds:
for every  $\omega\in H_{0}^{1}(d\nu)$,
\begin{equation*}
\left\| \omega \right\|_{L^{2}(d\nu)}^{2}\ \leq\  \frac{N-1}{N}\,
\int_{\Omega} \big({\rm{Ric}}_{V,N}\big)^{-1}(\nabla \omega,\nabla \omega)~d\nu\,.
\end{equation*}
\end{itemize}
\end{corollary}

\noindent 
Note that $\frac1N$ appears here as a natural parameter and that  $N\in(-\infty,0]\cup[n,+\infty]$ is equivalent to  $\frac{1}{N}\in [-\infty,\frac1n]$ with the convention $\frac{1}{0}=-\infty$.\\

\noindent
This result corresponds to the
 cases (1) and (2) of
 Theorem~1.2 in the recent article \cite{KoMi} to which we also refer
 for more details and references concerning the $N$-dimensional Bakry-\'Emery tensor
and its connections with the Bakry-\'Emery operators $\Gamma$ and $\Gamma_{2}$
 (see \eqref{eq.Gamma} and \eqref{eq.Gamma2} in the following section, and also \cite{BaGeLe}). The authors derive
these formulas from the so-called generalized Reilly formula
stated in Theorem~1.1 there, which somehow generalizes, in the weighted space setting,
 the statement given by Theorem~\ref{th.Witten-Bochner} 
 when  $p=1$ and $\omega$ has the form $d_{f}\eta$,
 to arbitrary $\omega=d_{f}\eta$ which are not assumed tangential nor normal.
 We also mention the related work \cite{KoMi2} of the same authors.\\ 
 
 \noindent
 Note lastly that for $N> n$, Corollary~\ref{co.main}
 does not provide any improvement in comparison with
 Corollary~\ref{co.BL} in the semiclassical setting, that is when $V$ is replaced by 
 $\frac Vh$ where $h\to 0^{+}$,
because of the term $-\frac{1}{(N-n)\,h^{2}}\,dV\otimes dV$
involved in ${\rm{Ric}}_{\frac Vh,N} $ (see indeed \eqref{eq.BakryEmeryCurvature}).

\subsection{Plan of the paper}

In the following section, we recall general definitions and properties related 
to the Riemannian structure and to
the Witten and weighted Laplacians. We then give the basic properties of the self-adjoint realizations
$\Delta_{f}^{\mathbf t,(p)}$ and $\Delta_{f}^{\mathbf n,(p)}$ in Section~\ref{se.Wittenavecbord}.
Lastly, in Section~\ref{se.proofs}, we prove Theorem~\ref{th.Witten-Bochner},
Theorem~\ref{th.main}, and Corollary~\ref{co.main}.

\section{Geometric setting}
\label{se.Geo}

\subsection{General definitions and properties}

 \noindent Let us begin with the notion of local orthonormal frame that will be frequently used in the sequel. 
 A local orthonormal frame on some open set $U\subset \Omega$
 is a
  family $(E_{1},\dots,E_{n})$ of smooth sections of $T\Omega$ defined on $U$
   such that
  $$
  \forall\,i,j\,\in\,\{ 1,\dots,n\}\,,\ \forall x\,\in\, U\ ,\quad \langle \,E_{i},E_{j}\,\rangle_{x}\ =\ \delta_{i,j}\,.
  $$
 According for example to \cite[Definition~1.1.6]{Sch} and to the related remarks,
 it is always possible to cover $\Omega$ with a finite family (since $\Omega$ is compact) of opens sets $U$'s such that
 there exists a local orthonormal frame $(E_{1},\dots,E_{n})$ on each $U$. Such a covering
 is called a nice cover of $\Omega$.\\
 
 \noindent
 The outgoing normal vector field will be denoted by $\vec n$
 and the orientation is chosen such that
 $$
 \mu_{\partial\Omega}\ =\ \mathbf i_{\vec n}\,\mu\,,
 $$
where   $\mathbf{i}$ denotes the interior product.
 Owing to the Collar Theorem stated in \cite[Theorem~1.1.7]{Sch}, the vector field 
 $\vec n
 \in \mathcal C^{\infty}(\partial\Omega,T\Omega\big|_{\partial\Omega})$
 can be extended to a smooth vector field on a neighborhood of the boundary $\partial\Omega$.
 Moreover, taking maybe a finite refinement of  a  nice cover of  $\Omega$ as defined previously,
 one can always assume that the local orthonormal frame $(E_{1},\dots,E_{n})$ corresponding to any of its elements $U$ meeting $\partial\Omega$ is such that $E_{n}\big|_{\partial\Omega}=\vec n$. 
\\

\noindent For any $\omega\in
\Lambda^{p}\mathcal{C}^{\infty}$, the tangential part of $\omega$ on $\partial\Omega$ is the form
 $\mathbf{t}\omega\in\mathcal{C}^{\infty}(\partial \Omega,
\Lambda^{p}T^{*}\Omega\big|_{\partial\Omega})$ defined by:
\begin{equation}
\label{eq.def-tan}
\forall \sigma \in \partial \Omega\ ,\quad 
 (\mathbf{t}\omega)_{\sigma}(X_{1},\ldots,
X_{p})\ :=\ \omega_{\sigma}(X_{1}^{T},\ldots, X_{p}^{T})\,,
\end{equation}
with the decomposition $X_{i}=X_{i}^{T}\oplus x_{i}^{\perp}\vec n_{\sigma}$ into the tangential 
and normal components to $\partial\Omega$ at $\sigma$.
More briefly, it holds $
\mathbf{t}\omega = \mathbf{i}_{\vec n}(\vec n^{\flat}\wedge \omega)$.
The normal part of $\omega$ on $\partial\Omega$ is then defined by:
\begin{equation}
\label{eq.def-nor}
\mathbf{n}\omega\ :=\ \omega|_{\partial\Omega}-\mathbf{t}\omega
\ =\ 
\vec n^{\flat}\wedge(\mathbf{i}_{\vec n} \omega)
 \quad
\in\  \mathcal{C}^{\infty}(\partial \Omega,\Lambda^{p}T^{*}\Omega\big|_{\partial\Omega}).
\end{equation}
Here and in the sequel, the notation  $\flat: X\mapsto X^{\flat}$ stands for the inverse isomorphism
of the canonical isomorphism $\sharp:\xi \mapsto \xi^{\sharp}  $  from
$T^{*}\Omega$ onto $T\Omega$ (defined by the relation $\langle\xi^{\sharp},X\rangle:= \xi(X)$ for every $X\in T\Omega$).\\

\noindent
For a (real) smooth function $f$ and a smooth vector field $X$,
we will use the notation
$$
\nabla_{X}f\ :=\ X\cdot f\ =\ df(X)\,,
$$  
the normal derivative of $f$ along the boundary being in particular defined
by
\begin{equation}
\label{eq.def-normal-der}
\partial_{n}f\ : =\ 
\left\langle\, \nabla f,\vec n\,\right\rangle\ =\ \nabla_{\vec n}f\,.
\end{equation}
We will also denote by $\nabla:\mathcal C^{\infty}(\Omega,T\Omega)\times
\mathcal C^{\infty}(\Omega,T\Omega)\to \mathcal C^{\infty}(\Omega,T\Omega)$ the Levi-Civita connection on $\Omega$ and by $\nabla_{X}(\cdot)$  the covariant derivative
 (in the direction of $X$) of vector fields as well as the induced
 covariant derivative
on $\Lambda^{p}T^{*}\Omega$. 
The second covariant derivative
 is then the bilinear mapping
on $T\Omega$ 
 defined, for $X,Y\in\mathcal C^{\infty}(\Omega,T\Omega)$ by
 $$
 \nabla^{2}_{X,Y}\ :=\ \nabla_{X}\nabla_{Y}-\nabla_{\nabla_{X}Y}\,.
 $$
When $f$ is a smooth function, 
$
\nabla^{2}_{X,Y} \,f
$
is simply the Hessian of $f$. It is in this case a symmetric bilinear form
and has the simpler writing
\begin{equation}
\label{eq.de-Hess}
\Hess f(X,Y)\ :=\ \nabla^{2}_{X,Y} \,f\ =\ (\nabla_{X} \,df)(Y)\ =\ \langle \nabla_{X}\nabla f,Y \rangle\,.
\end{equation}
The Bochner Laplacian $\Delta_{B}$
on $\mathcal{C}^{\infty}(\Omega, \Lambda
T^{*}\Omega)$
is defined as minus the trace of the bilinear mapping $(X,Y)\mapsto \nabla^{2}_{X,Y}$.
For any local orthonormal frame $(E_{1},\dots,E_{n})$ on $U\subset \Omega$,
$\Delta_{B}$ is in particular given on $U$ by
\begin{equation}
\label{eq.Bochner}
\Delta_{B}\ =\ -\sum_{i=1}^{n}\big(\nabla_{E_{i}}\nabla_{E_{i}}-\nabla_{\nabla_{E_{i}}E_{i}}\big)\,.
\end{equation}

\noindent
The Hodge and Bochner Laplacians $\Delta^{(p)}$ and $\Delta_{B}^{(p)}$
are   related
by the Weitzenb\"ock formula:
  there exists a smooth bundle symmetric endormorphism
 ${{\rm{Ric}}}^{(p)}$ belonging to $\mathcal L(\Lambda^{p}T^{*}\Omega)$
 such that (see \cite[p.~26]{Sch} where the opposite convention of sign is adopted)
 \begin{equation}\label{eq.Weitzen}
 \Delta^{(p)}_{B}\ =\ \Delta^{(p)}-{\rm{Ric}}^{(p)}\,.
 \end{equation}
This operator vanishes on $0$-forms (i.e. on functions)
and ${\rm{Ric}}^{(1)}$ is the element of $\mathcal L(\Lambda^{1} T^{*}\Omega)$
canonically identified with the Ricci tensor $\rm{ Ric}$.
We recall 
that  $\rm{ Ric}$
is 
the symmetric $(0,2)$-tensor defined, for $X,Y\in T\Omega$,  
by
\begin{equation}
\label{eq.Ricci}
{\rm{Ric}}(X,Y)\ :=\ \Tr\big(\,Z\longmapsto R(Z,X)Y\,\big)\,,
\end{equation}
where $R$ denotes the Riemannian curvature tensor which is defined,
for every $X,Y,Z\in T\Omega$, by
\begin{equation}
\label{eq.Riemann}
R(X,Y) Z\ :=\ \big(\nabla^{2}_{X,Y}-\nabla^{2}_{Y,X}\big)Z\ =\ \nabla_{X}\nabla_{Y}Z-\nabla_{Y}\nabla_{X}Z-\nabla_{[X,Y]}Z\,.
\end{equation}

\noindent
More generally, 
for any  local orthonormal frame $(E_{1},\dots,E_{n})$ on $U\subset \Omega$,
${\rm{Ric}}^{(p)}$ is defined on $U$ for any $p\in\{1,\dots,n\}$ by
\begin{align}
\nonumber
\big(&{\rm{Ric}}^{(p)}\omega\big)(X_{1},\dots,X_{p})\\
\label{eq.Ric-p}
& := -\sum_{i=1}^{n}\sum_{j=1}^{p}\Big(\big(R(E_{i},X_{j})\big)^{(p)}\omega\Big)(X_{1},\dots,X_{j-1},E_{i},X_{j+1},\dots,X_{p})\,,
\end{align}
where $\big(R(E_{i},X_{j})\big)^{(1)}\in
\mathcal L(\Lambda^{1} T^{*}\Omega)$
is canonically identified with $R(E_{i},X_{j})$
via $\Big(\big(R(E_{i},X_{j})\big)^{(1)}\omega\Big)(X)=\omega\big(\,R(E_{i},X_{j})X\,\big)$
 and 
$$\big(R(E_{i},X_{j})\big)^{(p)}\ =\ \Big(\big(R(E_{i},X_{j})\big)^{(1)}\Big)^{(p)}\,,$$
where
for  any $A\in\mathcal L(\Lambda^{1}T^{*}\Omega)$, $(A)^{(p)}$ is the element of $\mathcal L(\Lambda^{p}T^{*}\Omega)$
satisfying the following relation on decomposable $p$-forms:
\begin{equation}
\label{eq.Ap}
(A)^{(p)}\big(
\omega_{1}\wedge\cdots\wedge \omega_{p}\big)\ =\ 
\sum_{i=1}^{p}  \omega_{1}\wedge\cdots \wedge A\omega_{i}\wedge\cdots\wedge
\omega_{p}\,. 
\end{equation}

\noindent
To end up this part,
we recall the definition of
the  second fundamental form of $\partial\Omega$
before defining the operators  $\mathcal K_{\mathbf b}^{(p)}$, 
$\mathbf b\in\{\mathbf t,\mathbf n\} $, involved in Theorem~\ref{th.Witten-Bochner}
and in its corollaries.
The  second fundamental form $\mathcal K_{2}$ of $\partial\Omega$
 is the symmetric bilinear mapping defined by
 \begin{equation}
 \label{de.secondff}
 \mathcal K_{2}\ :\ \begin{array}{ccc} T\partial\Omega\times T\partial\Omega &\longrightarrow & T\Omega\,\big|_{\partial\Omega}\\
 (U,V)&\longmapsto& (\nabla_{U} V)^{\perp}:=\langle \nabla_{U} V, \vec n \rangle\, \vec n\end{array}
 \end{equation}
and it satisfies: 
$$\forall\,(U,V)\,\in\,T\partial\Omega\times T\partial\Omega\ ,\quad  \langle \mathcal K_{1}(U), V \rangle\,\vec n
\ =\ \mathcal K_{2}(U,V)\,,$$
where $\mathcal K_{1}\in\mathcal L(T\partial\Omega) $ is the shape operator of $\partial\Omega$ which is defined by:
\begin{equation}
\label{eq.shape}
\forall\,U\,\in\,T\partial\Omega\ ,\quad \mathcal K_{1} (U)\ :=\ -\nabla_{U}\,\vec n\,.
\end{equation}
The mean curvature of $\partial\Omega$
is defined as the trace of the bilinear mapping $(U,V)\mapsto \langle\mathcal K_{2}(U,V),\vec n\rangle$
or equivalently as the trace of the shape operator $\mathcal K_{1}$.
Note also that with our choice of orientation for $\vec n$, $\Omega$ is locally convex  iff $\langle\mathcal K_{2}(\cdot,\cdot),\vec n\rangle$ (or equivalently $\mathcal K_{1}$, in the sense of quadratic forms)
is nonpositive.\\

\noindent Lastly,  the smooth 
 bundle endormophisms $
\mathcal K_{\mathbf b}^{(p)}\in \mathcal L(\Lambda^{p}T^{*}\Omega\big|_{\partial\Omega})
$, where $\mathbf b\in\{\mathbf n,\mathbf t\}$ and $p\in\{0,\dots,n\}$,
are defined by means of $\mathcal K_{1}$ and $\mathcal K_{2}$ as follows:
\begin{enumerate}
\item For any $p\in\{0,\dots,n\}$, 
$
\mathcal K_{\mathbf n}^{(p)}\in \mathcal L(\Lambda^{p}T^{*}\Omega\big|_{\partial\Omega})
$ 
vanishes on $0$-forms and:
\begin{itemize}
\item[i)] for any  $\omega\in\Lambda^{1}T^{*}\Omega$, $\mathcal K_{\mathbf n}^{(1)}\omega$ is tangential and
\begin{equation}
\label{eq.Kn1}
 (\mathcal K_{\mathbf n}^{(1)}\omega)(X^{T}+x^{\perp}\vec n)=-\omega\big(\,\mathcal K_{1}(X^{T})\,\big)
 =\omega\big(\,\nabla_{\!X^{T}}\,\vec n\,\big)\,,
\end{equation}
where $\mathcal K_{1}$ is the shape operator defined in \eqref{eq.shape},
\item[ii)] for any $p\in\{1,\dots,n\}$ and $\omega\in\Lambda^{p}T^{*}\Omega$,
$\mathcal K_{\mathbf n}^{(p)}\omega$ is tangential and for any $X^{T}_{1},\dots,X^T_{p}\in T\partial\Omega$,
\begin{equation}
\label{eq.Knp}
\big(\mathcal K_{\mathbf n}^{(p)}\omega\big)(X_{1}^{T},\dots,X_{p}^{T})\ =\ \big((\mathcal K_{\mathbf n}^{(1)})^{(p)}\omega\big)(X_{1}^{T},\dots,X_{p}^{T})\,,
\end{equation}
where the notation $(A)^{(p)}$ has been defined in \eqref{eq.Ap}.
\end{itemize}
\item For any $p\in\{0,\dots,n\}$, 
$
\mathcal K_{\mathbf t}^{(p)}\in\mathcal L(\Lambda^{p}T^{*}\Omega\big|_{\partial\Omega})
$ 
vanishes on $0$-forms and:
\begin{itemize}
\item[i)] for any  $\omega\in\Lambda^{1}T^{*}\Omega$, $\mathcal K_{\mathbf t}^{(1)}\omega$ is normal and
\begin{equation}
\label{eq.Kt1}
 (\mathcal K_{\mathbf t}^{(1)}\omega)(X^{T}+x^{\perp}\vec n)=-x^{\perp}\,\Tr(\mathcal K_{1})\,\omega\big(\,\vec n\,\big)\,,
\end{equation}
\item[ii)] for any $p\in\{1,\dots,n\}$ and $\omega\in\Lambda^{p}T^{*}\Omega$,
$\mathcal K_{\mathbf t}^{(p)}\omega$ is normal and for any
local orthonormal frame $(E_{1},\dots,E_{n})$ on $U\subset \Omega$ 
such that $E_{n}\big|_{\partial\Omega}=\vec n$
(with $U\cap\partial\Omega\neq\emptyset$) and
 $X^{T}_{1},\dots,X^T_{p}\in T\partial\Omega$,
 we have on $U\cap\partial\Omega$:
 \begin{align}
 \nonumber
\big(\,\mathcal K_{\mathbf t}^{(p)}\omega\,\big)&(\vec n,X^{T}_{1},\dots,X^{T}_{p-1})\\
\label{eq.Ktp}
&\ :=\ 
-\sum_{i=1}^{n-1}\Big(\big(\mathcal K_{2}(E_{i},\cdot )\big)^{(p)}\omega\Big)(E_{i},X_{1}^{T},\dots,X^{T}_{p-1})
\,,
\end{align}
where 
$\big(\mathcal K_{2}(E_{i},\cdot )\big)^{(p)}=\Big(\big(\mathcal K_{2}(E_{i},\cdot )\big)^{(1)}\Big)^{(p)}$
and $$\Big(\big(\mathcal K_{2}(E_{i},\cdot )\big)^{(1)}\omega\Big)(X)\ =\ 
\omega\big(\,\mathcal K_{2}(E_{i},X)\,\big)\,.$$
\end{itemize}
\end{enumerate}
Note that  the point 2.ii) is nothing but the statement of 2.i) when $p=1$.\\

\subsection{Witten and weighted Laplacians}

\noindent Using the following relations dealing with exterior and interior products 
(respectively denoted by $\wedge$ and $\mathbf{i}$),
gradients, and
Lie derivatives (denoted by $\mathcal{L}$),
\begin{align}
\left(df\wedge\right)^{*}&=\mathbf{i}_{\nabla f}\quad  \text{as bounded operators in }L^{2}(\Omega,\Lambda^{p}T^{*}\Omega\,)\,,\label{eq.df0}\\
\label{eq.df1}
  d_{f}&=d+df\wedge\quad \text{and}\quad d_{f}^{*}=d^{*}+\mathbf{i}_{\nabla f}\, ,\\
\label{eq.df3}
\mathcal{L}_{X}&=d\circ\mathbf{i}_{X}+\mathbf{i}_{X}\circ d\quad
\text{and}\quad
\mathcal{L}^{*}_{X}=d^{*}\circ (X^{\flat}\wedge\cdot)+X^{\flat}\wedge d^{*}\,,
\end{align}
the Witten Laplacian $(d_{f}+d_{f}^{*})^{2}$ has the form
\begin{equation}
\label{eq.df4}
 \Delta_{f}\ =\ (d+d^{*})^{2}+\left|\nabla f\right|^{2}+
\left(\mathcal{L}_{\nabla f}+\mathcal{L}_{\nabla f}^{*}\right)\,.
\end{equation}
\ 
\\
\noindent
When acting on $0$-forms,  
the Witten Laplacian is then simply given by
$$
\Delta^{(0)}_{f}\ =\ \Delta+|\nabla f|^{2}+\Delta f
$$ 
and is 
unitarily equivalent to the 
following operator 
acting on the weighted space  $L^{2}(e^{-2f}d\mu)$,
sometimes referred to as the weighted Laplacian (or Bakry-\'Emery Laplacian)  
in the literature (see e.g. \cite{KoMi}), 
\begin{equation*}
\label{eq.LVh}
L^{(0)}_{V}\ :=\ \Delta+\nabla V\cdot\nabla\qquad\text{where}\qquad V\ :=\ 2f\,. 
\end{equation*}  
More precisely, it holds
$$
\Delta^{(0)}_{f}\ =\ e^{-\frac{V}{2}} \,L^{(0)}_{V} \,e^{\frac{V}{2}}\,.
$$ 
The operator $L^{(0)}_{V}$ has consequently a natural 
supersymmetric extension  on the algebra of 
differential
forms, acting 
in the weighted space $\Lambda L^{2}(e^{-2f}d\mu)$, which is simply defined
for any $p\in\{0,\dots,n\}$ by the formula
\eqref{eq.LVh-p} that we recall here:
\begin{equation*}
\label{eq.LVh-p'}
L^{(p)}_{V}\ :=\ e^{f}\,\Delta^{(p)}_{f}\,e^{-f}\qquad\text{where}\qquad V\ :=\ 2f\,.
\end{equation*}
To connect more precisely to  the literature dealing with the Bakry-\'Emery theory of diffusion processes
(see \cite{BaGeLe} for an overview), 
the
  operators $L^{(0)}_{V}$ and $L^{(1)}_{V}$  are related to the carr\'e du champ operator of
Bakry-\'Emery $\Gamma$ and to its iteration
$\Gamma_{2}$
 via the relations 
\begin{equation}
\label{eq.Gamma}
\int_{\Omega}\Gamma(\omega)
~e^{-2f}d\mu\, =\, \int_{\Omega}\big( L^{(0)}_{V}\,\omega\big)\,\omega~e^{-2f}d\mu
 = \int_{\Omega}\langle d \omega,d\omega \rangle_{\Lambda^{1}}\,~e^{-2f}d\mu
\end{equation}
and
\begin{equation}
\label{eq.Gamma2}
\int_{\Omega}\!\!\Gamma_{2}(\omega)e^{-2f}d\mu = \!\!
\int_{\Omega}\!\!\big( L^{(0)}_{V}\,\omega\big)^{2}e^{-2f}d\mu
 = \!\int_{\Omega}\langle L^{(1)}_{V}\,d \omega,d\omega \rangle_{\Lambda^{1}}e^{-2f}d\mu,
\end{equation}
where $\omega$ is a smooth function supported in $\Omega\setminus\partial\Omega$
(see in particular \cite{BaGeLe} for many details and references about this notion).\\

\noindent
Coming back to the Witten Laplacian, 
we have  the following
formula: 
\begin{equation}
\label{eq.Hess}
 \Delta^{(p)}_{f}\ =\ (d+d^{*})^{2}\,+\,\left|\nabla f\right|^{2}\,+\,2\,\Hess^{\!(p)}\!f +\,\Delta f \,.
\end{equation}
This relation is not very common in the  literature 
dealing with semiclassical Witten Laplacians --
i.e. where one studies $h^{2}\Delta_{\frac fh}$ at the limit $h\to0^{+}$ --
which motivated this work, 
at least when $\Omega$
is not flat.
We find generally there the formula \eqref{eq.df4} 
  (see e.g. \cite{HelNi1,Lep} and references therein)
  and we 
 thus give a proof below (see also \cite{Jam}
for another proof). Let us, before, specify the sense of \eqref{eq.Hess}.
There, $\Hess^{\!(0)}\!f=0$, $\Hess^{\!(1)}\!f$ is the element of
$\mathcal L(\Lambda^{1}T^{*}\Omega)$ canonically identified
with $\Hess f$,
and $\Hess^{\!(p)}\!f$
is  the bundle symmetric endomorphism of
$\Lambda^{p}T^{*}\Omega$
defined by 
\begin{equation}
\label{eq.Hess-p}
\Hess^{\!(p)}\!f\ :=\ \big(\Hess^{\!(1)}\!f\big)^{(p)}\quad\text{(see \eqref{eq.Ap} for the meaning of $(A)^{(p)}$)}\,.
\end{equation}
Denoting also by $\Hess\! f$ the bundle symmetric endomorphism
of $T\Omega$ defined by $\langle\Hess\! f\, X,Y\rangle:=\Hess \!f(X,Y) $
(i.e. by $\Hess\! f\, X=\nabla_{X}\nabla f$), remark that we have
for any $p\in\{1,\dots,n\}$
and
$\omega \in \Lambda^{p}T^{*}\Omega$:
\begin{equation}
\label{eq.Hess-p'}
\Hess^{\!(p)}\!f\,\omega (X_{1},\dots,X_{p})\ =\ \sum_{i=1}^{p}\omega(X_{1},\dots,\Hess \!f\, X_{i},\dots,X_{p})\,. 
\end{equation}

\begin{proof}[Proof of formula \eqref{eq.Hess}:]
Let us first recall that the covariant  derivative $\nabla_{X}$
on $\Lambda^{p}T^{*}\Omega$
 induced by the Levi-Civita connection
 is defined
 by
\begin{align}
\nonumber
(\nabla_{X}\omega)(Y_{1},\dots,Y_{p}) &\ :=\ 
\nabla_{X} \big(\omega(Y_{1},\dots,Y_{p})\big)\\
\label{eq.nabla-form}
&\qquad\qquad-\,\sum_{k=1}^{p}\omega(Y_{1},\dots,\nabla_{X}Y_{k},\dots,Y_{p})
\end{align}
and satisfies in particular the relations
\begin{equation}
\label{eq.nabla-lambda}
\nabla_{X}\big(\langle \omega,\eta\rangle_{\Lambda^{p}}\big)=\langle\nabla_{X} \omega,\eta\rangle_{\Lambda^{p}}
+\langle \omega,\nabla_{X}\eta\rangle_{\Lambda^{p}}
\end{equation}
and
\begin{equation}
\label{eq.nabla-wedge}
\nabla_{X} \big(\omega_{1}\wedge\omega_{2}\big)=(\nabla_{X} \omega_{1})\wedge \omega_{2}
+\omega_{1}\wedge(\nabla_{X} \omega_{2})\,.
\end{equation}

\noindent
The differential $d$ and $\nabla$ are moreover related by the relation
\begin{equation}
\label{eq-nabla-d}
d\omega (X_{0},\dots,X_{p})\ =\ \sum_{k=0}^{p}(-1)^{k}(\nabla_{X_{k}}\omega)(X_{0},\dots,\dot{X_{k}},\dots,X_{p})\,,
\end{equation}
where
$\omega \in \Lambda^{p}T^{*}\Omega$ and
 the notation $\dot{X_{k}}$ means that $X_{k}$ has been removed from the parenthesis.
Furthermore,
if $(E_{1},\dots,E_{n})$ is a local orthonormal frame on $U\subset \Omega$,
the codifferential $d^{*}$
is given there by
\begin{equation}
\label{eq.d*}
d^{*}\ =\ -\sum_{i=1}^{n}\mathbf{i}_{E_{i}}\nabla_{E_{i}}\,.
\end{equation}

\noindent
Hence, we deduce  from  \eqref{eq.Hess-p'} and from the relation relating $\mathcal L_{X}$ and $\nabla_{X}$,
$$
(\mathcal L^{(p)}_{X}\omega) (X_{1},\dots, X_{p})\ =\ (\nabla_{X}\omega)(X_{1},\dots,X_{p})+\sum_{i=1}^{p}\omega(X_{1},\dots,\nabla_{X_{i}}X,\dots, X_{p})
$$
which arises from \eqref{eq.nabla-form}, \eqref{eq-nabla-d}, and \eqref{eq.df3},
the following equality:
\begin{equation}
\label{eq.Lp}
\mathcal L^{(p)}_{\nabla f}\  =\ \nabla_{\nabla f} \,+\,\Hess^{\!(p)}\!f\,.
\end{equation}
Taking now a local orthonormal frame $(E_{1},\dots,E_{n})$ on an open set 
$U\subset \Omega$, we deduce  from \eqref{eq.nabla-wedge}, \eqref{eq.d*}, and \eqref{eq.df3}
the following relations (on $U$):
\begin{align}
\nonumber
\mathcal L^{*,(p)}_{\nabla f}\omega
&\ =\ 
\sum_{i=1}^{n}\Big((\nabla_{E_{i}}df)\wedge \mathbf i_{E_{i}}\omega-df(E_{i})\nabla_{E_{i}}\omega
-\big(\nabla_{E_{i}}df(E_{i})\big)\Big)\omega\\
\label{eq.Lp*1}
&\ =\ 
-\nabla_{\nabla f}\,\omega + (\Delta f)\omega + \sum_{i=1}^{n}(\nabla_{E_{i}}df)\wedge \mathbf i_{E_{i}}\omega\,.
\end{align}
Lasty, we have 
\begin{align}
\nonumber
\sum_{i=1}^{n}\big((\nabla_{E_{i}}df)\wedge &\mathbf i_{E_{i}}\omega\big)(X_{1},\dots,X_{p})\\
\nonumber
&\ =\ \sum_{i=1}^{n} \sum_{k=1}^{p}(-1)^{k+1} (\nabla_{E_{i}}df)(X_{k})(\mathbf i_{E_{i}}\omega)(X_{1},\dots,\dot{X_{k}},\dots,X_{p})\\
\nonumber
&\ =\ \sum_{k=1}^{p}(-1)^{k+1} \omega(\Hess f\, X_{k},X_{1},\dots,\dot{X_{k}},\dots,X_{p})\\
\label{eq.Lp*2}
&\ =\ \sum_{k=1}^{p} \omega(X_{1},\dots,\Hess f\, X_{k},\dots,X_{p})
\end{align}
and formula \eqref{eq.Hess} for the Witten Laplacian
then follows from \eqref{eq.df4} and \eqref{eq.Lp}--\eqref{eq.Lp*2}.
\end{proof}

\section{Self-adjoint realizations of the Witten Laplacian}
\label{se.Wittenavecbord}

\noindent
In the sequel, we will use
for any $(\omega,\eta)\in \big(\Lambda^{p} H^{1}\big)^{2} $
the more compact notation
$$
\mathcal D_{f}^{(p)}(\omega,\eta)\ :=\ \langle d_{f}\omega,d_{f}\eta\rangle_{\Lambda^{p+1}L^{2}}+\langle d^{*}_{f}\omega,d^{*}_{f}\eta\rangle_{\Lambda^{p-1}L^{2}}
$$
as well as
$$
\mathcal D_{f}^{(p)}(\omega)\ :=\ \mathcal D_{f}^{(p)}(\omega,\omega)
$$
and
$$
\mathcal D^{(p)}(\omega,\eta)\ :=\ \mathcal D_{0}^{(p)}(\omega,\eta)\quad \text{and}\quad
\mathcal D^{(p)}(\omega):=\mathcal D^{(p)}(\omega,\omega)\,.
$$
Let us also recall, for $\mathbf b \in \{\mathbf n,\mathbf t\}$, the definition of $\Lambda^{p}H^{1}_{\mathbf{b}}$ given in \eqref{eq.domain-Delta-Q}:
\begin{equation*}\label{introespace}
\Lambda^{p}H^{1}_{\mathbf{b}}\ =\ 
\left\{\omega \in \Lambda^{p}H^{1}\,,\ 
\mathbf{b}\omega = 0\ \text{on}\ \partial\Omega\right\}\,.
\end{equation*}
In particular,  $\Lambda^0H^{1}_{\mathbf{n}}=H^{1}_{\mathbf{n}}$ 
is simply $H^{1}(\Omega)$
while
$H^{1}_{\mathbf{t}}=H_{0}^{1}(\Omega)$.
Moreover,  since the boundary $\partial \Omega$ is smooth,
the space
$$
\Lambda^{p}\mathcal{C}^{\infty}_{\mathbf{b}}\ :=\ 
\left\{\omega \in \Lambda^{p}\mathcal{C}^{\infty}\,,\ 
\mathbf{b}\omega = 0\ \text{on}\ \partial\Omega\right\}
$$
is dense in $\big(\Lambda^{p}H^{1}_{\mathbf{b}}\,,\,\|\cdot\|_{\Lambda^{p} H^{1}}\big)$.\\

\noindent
 The following lemma states two Green's identities comparing 
$\mathcal D^{(p)}(\cdot)$ and  $\mathcal D_{f}^{(p)}(\cdot)$
 on the space of tangential or normal $p$-forms. We refer to
  \cite[Section~2.3]{HelNi1} and \cite[Section~2.2]{Lep} for a proof.

\begin{lemma}
  \label{le.intbypart}
We have the two following identities:\begin{enumerate}
\item[i)] for any $\omega\in \Lambda^{p}H_{\mathbf{t}}^{1}$,
\begin{multline}
  \label{eq.Green-Dirichlet}
\mathcal D^{(p)}_{f}(\omega)=
 \mathcal D^{(p)}(\omega)+ \left\|\,\left|\nabla f\right|\omega\right\|^{2}_{\Lambda^{p}L^{2}}
+\langle(\mathcal{L}_{\nabla f}+\mathcal{L}_{\nabla
  f}^{*})\omega,\omega\rangle_{\Lambda^{p}L^{2}}\\
+\int_{\partial \Omega}\langle\omega,\omega
\rangle_{\Lambda^{p}}\,
\partial_{n}f~d \mu_{\partial \Omega}\,,
\end{multline}
\item[ii)] for any $\omega\in \Lambda^{p}H_{\mathbf{n}}^{1}$,
\begin{multline}
   \label{eq.Green-Neumann}
\mathcal D^{(p)}_{f}(\omega)=
 \mathcal D^{(p)}(\omega)+ \left\|\,\left|\nabla f\right|\omega\right\|^{2}_{\Lambda^{p}L^{2}}
+\langle(\mathcal{L}_{\nabla f}+\mathcal{L}_{\nabla
  f}^{*})\omega,\omega\rangle_{\Lambda^{p}L^{2}}\\
-\int_{\partial \Omega}\langle\omega,\omega
\rangle_{\Lambda^{p}}\,\partial_{n}f~d \mu_{\partial \Omega}\,.
\end{multline}
\end{enumerate}
\end{lemma}

\noindent 
We now compile in  the following proposition basic facts about Witten Laplacians
on manifolds with boundary proven in \cite[Section~2.4]{HelNi1} and  in \cite[Section~2.3]{Lep}.
\begin{proposition}
  \label{pr.Witten}
  \begin{enumerate}
  \item[i)] For $\mathbf b \in \{\mathbf n,\mathbf t\}$ and $p\in\{0,\dots,n\}$,
  the nonnegative quadratic form $\omega\to\mathcal{D}^{(p)}_{f,h}(\omega)$
is closed on $\Lambda^{p}H^{1}_{\mathbf{b}}$. Its associated self-adjoint Friedrichs
extension
is denoted by $\big(\Delta_{f}^{\mathbf{b},(p)}\,,\, D(\Delta_{f}^{\mathbf{b},(p)})\big)$.
\item[ii)]  For $\mathbf b \in \{\mathbf n,\mathbf t\}$
 and $p\in \left\{0,\ldots, n\right\}$,  the domain of $\Delta_{f}^{\mathbf{b},(p)}$
 is given by
 $$
D(\Delta_{f}^{\mathbf{b},(p)})\ =\ \left\{u\in \Lambda^{p}H^{2},\ 
  \mathbf{b}\omega=0\,,\  \mathbf{b}d^{*}_{f}\omega=0 \ \text{and}\ \mathbf{b}d_{f}\omega=0
  \ \text{on}\ \partial\Omega\right\}\!\!.
$$
We have moreover: 
$$
\forall\, \omega\, \in\, D(\Delta_{f}^{\mathbf b,(p)})\,,\quad
\Delta_{f}^{\mathbf b, (p)}\omega\ =\ \Delta^{(p)}_{f}\omega\quad\text{in}\quad \Omega
$$
and the equalities $\mathbf{n}d^{*}_{f}\omega=0$ and $\mathbf{t}d_{f}\omega=0$
are actually satisfied for any $\omega\in\Lambda^{p}H^{2}\cap \Lambda^{p}H^{1}_{\mathbf{b}}$.

\item[iii)] For $\mathbf b \in \{\mathbf n,\mathbf t\}$
 and $p\in \left\{0,\ldots, n\right\}$,  $\Delta_{f}^{\mathbf{b},(p)}$
has a compact resolvent.
\item[iv)] 
For $\mathbf b \in \{\mathbf n,\mathbf t\}$ and $p\in \left\{0,\ldots, n\right\}$,
the following commutation relations hold for any $v\in\Lambda^{p} H^{1}_{\mathbf{b}} $:\\
-- for every $z\in \varrho(\Delta_{f}^{\mathbf{b},(p)})\cap \varrho(\Delta_{f}^{\mathbf{b},(p+1)})$,
$$
(z-\Delta_{f}^{\mathbf{b},(p+1)})^{-1}\, d_{f}^{(p)}\,v\ =\ 
d_{f}^{(p)}\, (z-\Delta_{f}^{\mathbf{b},(p)})^{-1}\,v
$$
-- and for every $z\in \varrho(\Delta_{f}^{\mathbf{b},(p)})\cap \varrho(\Delta_{f}^{\mathbf{b},(p-1)})$,
$$
 (z-\Delta_{f}^{\mathbf{b},(p-1)})^{-1} \,d_{f}^{(p-1),*}\,v\ =\ 
d_{f}^{(p-1),*}\, (z-\Delta_{f}^{\mathbf{b},(p)})^{-1}\,v\,.
$$
\end{enumerate}
\end{proposition}

\noindent\\ In the spirit of the above point iv), we have also the following Witten-Hodge-decomposition
which will be useful when proving Corollary~\ref{co.main}:

\begin{proposition}
\label{pr.ortho-decomp}
For $\mathbf b \in \{\mathbf n,\mathbf t\}$ and $p\in \left\{0,\ldots, n\right\}$, it holds
\begin{align}
\label{eq.ortho-decomp}
\Lambda^{p}L^{2}\ &=\ \Ker \Delta_{f}^{\mathbf b,(p)}
\oplus^{\perp}
\Ran\, \big(d_{f}\big|_{\Lambda^{p-1}H^{1}_{\mathbf b} }\big)
\oplus^{\perp}
\Ran\,\big( d^{*}_{f}\big|_{\Lambda^{p+1}H^{1}_{\mathbf b} }\big)\\
\nonumber
\ &=:\ K^{\mathbf b,(p)}
\oplus^{\perp } 
R^{\mathbf b,(p)}
\oplus^\perp R^{*,\mathbf b,(p)}\,,
\end{align}
the spaces 
$R^{\mathbf b,(p)}$
 and 
 $R^{*,\mathbf b,(p)}$
being consequently closed in $\Lambda^{p}L^{2}$.
Denoting moreover by  $\pi_{R^{\mathbf b,(p)}}$ and  
$\pi_{R^{*,\mathbf b,(p)}}$ the  orthogonal projectors on these respective  spaces, the following relations hold
in the sense of unbounded operators:
\begin{equation}
\label{eq.ortho-decomp'}
\pi_{R^{\mathbf b,(p)}}\,\Delta_{f}^{\mathbf b,(p)}\, \subset\, \Delta_{f}^{\mathbf b,(p)}\,\pi_{R^{\mathbf b,(p)}}
\ \text{and}
\  
\pi_{R^{*,\mathbf b,(p)}}\,\Delta_{f}^{\mathbf b,(p)}\, \subset\, \Delta_{f}^{\mathbf b,(p)}\,\pi_{R^{*,\mathbf b,(p)}}.
\end{equation} 
In particular, for $ A\in\{(K^{\mathbf b,(p)})^{\perp},R^{\mathbf b,(p)},R^{*,\mathbf b,(p)}\}$,
the unbounded operator 
$\Delta_{f}^{\mathbf{b},(p)}\big|_{A}$
 with domain
$D(\Delta_{f}^{\mathbf{b},(p)})\cap A$
is well defined, self-adjoint,  invertible on $A$,  and it holds  for every $v\in\Lambda^{p} H^{1}_{\mathbf{b}}\cap (K^{\mathbf b,(p)})^{\perp} $: 
\begin{align}
\nonumber
\big(\Delta_{f}^{\mathbf{b},(p+1)}\big|_{(K^{\mathbf b,(p+1)})^{\perp}}\big)^{-1}\, d_{f}\,v\ &=\ \big(\Delta_{f}^{\mathbf{b},(p+1)}\big|_{R^{\mathbf b,(p+1)}}\big)^{-1}\, d_{f}\,v\\
\label{eq.ortho-decomp''}
\ &=\ 
d_{f}\,\big(\Delta_{f}^{\mathbf{b},(p)}\big|_{(K^{\mathbf b,(p)})^{\perp}}\big)^{-1}\,v
\end{align}
and
\begin{align}
\nonumber
\big(\Delta_{f}^{\mathbf{b},(p-1)}\big|_{(K^{\mathbf b,(p-1)})^{\perp}}\big)^{-1}\, d^{*}_{f}\,v
\ &=\  
\big(\Delta_{f}^{\mathbf{b},(p-1)}\big|_{R^{*,\mathbf b,(p-1)}}\big)^{-1}\, d^{*}_{f}\,v
\\
\label{eq.ortho-decomp'''}
\ &=\ 
d^{*}_{f}\,\big(\Delta_{f}^{\mathbf{b},(p)}\big|_{(K^{\mathbf b,(p)})^{\perp}}\big)^{-1}\,v\,.
\end{align}
\end{proposition}

\begin{proof}
The orthogonality of the sum appearing in the r.h.s. of \eqref{eq.ortho-decomp}
follows easily from the distorted Green's formula valid for any $(\omega,\eta)\in  \Lambda^{p-1} H^{1}\times
\Lambda^{p} H^{1}$,
\begin{align}
\label{Green-distorted} 
\langle d_{f} \omega,\eta\rangle_{\Lambda^{p}L^{2}}
\ =\ 
\langle \omega,d_{f}^{*}\eta\rangle_{\Lambda^{p-1}L^{2}}
+\int_{\partial\Omega}\langle
\omega,\mathbf{i}_{\vec n}\eta\rangle_{\Lambda^{p}} \, d\mu_{\partial\Omega}\,,
\end{align}
which is a straightforward consequence of \eqref{eq.df0}--\eqref{eq.df1}
and of
the usual Green's formula:
\begin{align}
\langle d \omega,\eta\rangle_{\Lambda^{p}L^{2}}
\label{usual-Green'}
&\ =\ \langle \omega,d^{*}\eta\rangle_{\Lambda^{p-1}L^{2}}
+\int_{\partial\Omega}\langle
\omega,\mathbf{i}_{\vec n}\eta\rangle_{\Lambda^{p}} \, d\mu_{\partial\Omega}\,.
\end{align}
Moreover, since 
$\Delta_{f}^{\mathbf{b},(p)}$ has a compact resolvent, the self-adjoint operator 
$\tilde\Delta_{f}^{\mathbf{b},(p)}$
on $(K^{\mathbf{b},(p)})^{\perp}:=\Ker(\Delta_{f}^{\mathbf{b},(p)})^{\perp}$,
\begin{align*}
\tilde\Delta_{f}^{\mathbf{b},(p)}\ :=\ \Delta_{f}^{\mathbf{b},(p)}\big|_{(K^{\mathbf{b},(p)})^{\perp}}
:\ D(\Delta_{f}^{\mathbf{b},(p)})\cap \Ker(\Delta_{f}^{\mathbf{b},(p)})^{\perp}\longrightarrow \Ker(\Delta_{f}^{\mathbf{b},(p)})^{\perp}\,,
\end{align*}
is invertible
and hence any $u\in\Lambda^{p}L^{2}$ has the form 
\begin{equation}
\label{eq.decomp-u}
u\ =\ \pi_{f,\mathbf b}u\,+\, \Delta_{f}^{\mathbf b,(p)}v
\ =\ \pi_{f,\mathbf b}u\,+\, d_{f}\big(d^{*}_{f}v\big)\,+\,d^{*}_{f}\big(d_{f}v\big)\,,
\end{equation}
for some uniquely determined $v\in D(\Delta_{f}^{\mathbf{b},(p)})\cap(K^{\mathbf{b},(p)})^{\perp}$,
denoting by $\pi_{f,\mathbf b}=\pi^{(p)}_{f,h,\mathbf b}$ the orthogonal projection on $(K^{\mathbf{b},(p)})^{\perp}$. This implies \eqref{eq.ortho-decomp}.\\[0.1cm]
Let us now prove \eqref{eq.ortho-decomp'} and take then $u\in D(\Delta_{f}^{\mathbf b,(p)})$. 
It holds 
\begin{equation}
\label{eq.proj-Delta-u}
\pi_{R^{\mathbf b,(p)}}\,\Delta_{f}^{\mathbf{b},(p)}\,u\ =\ d_{f}\big(d^{*}_{f}u\big)
\ \ \text{and}\ \ 
\pi_{R^{*,\mathbf b,(p)}}\,\Delta_{f}^{\mathbf{b},(p)}\,u\ =\ d^{*}_{f}\big(d_{f}u\big)
\end{equation}
and, according to \eqref{eq.decomp-u}, we have moreover
\begin{equation}
\label{eq.proj-u}
\pi_{R^{\mathbf b,(p)}}u\ =\ d_{f}\big(d^{*}_{f}v\big)\quad\text{and}
\quad
\pi_{R^{*,\mathbf b,(p)}}u\ =\ d^{*}_{f}\big(d_{f}v\big)
\end{equation}
where $v=\big(\tilde\Delta_{f}^{\mathbf{b},(p)}\big)^{-1}(u-\pi_{f,\mathbf b}u)\in D(\Delta_{f}^{\mathbf{b},(p)})\cap(K^{\mathbf{b},(p)})^{\perp}$. 
Using now iv) of Proposition~\ref{pr.Witten}, we have for every $z\in\mathbf R$, $z<0$,
\begin{align}
\nonumber
d_{f}\,d^{*}_{f}\big(\Delta_{f}^{\mathbf{b},(p)}-z\big)^{-1}\,(u-\pi_{f,\mathbf b}u)
 =\, & 
\big(\Delta_{f}^{\mathbf{b},(p)}-z\big)^{-1} d_{f}\,d^{*}_{f}\,(u-\pi_{f,\mathbf b}u)\\
\label{eq.commut-z}
 \underset{z\to 0^{-}}{\longrightarrow}&\big(\tilde\Delta_{f}^{\mathbf{b},(p)}\big)^{-1} d_{f}\,d^{*}_{f}(u-\pi_{f,\mathbf b}u).
\end{align}
 Since moreover 
 $
 \big(\Delta_{f}^{\mathbf{b},(p)}-z\big)^{-1}(u-\pi_{f,\mathbf b}u)= v+z\,\big(\Delta_{f}^{\mathbf{b},(p)}-z\big)^{-1}v
 $,
 it also holds 
 \begin{align}
 \nonumber
 d_{f}\,d^{*}_{f}\big(\Delta_{f}^{\mathbf{b},(p)}-z\big)^{-1}\,(u-\pi_{f,\mathbf b}u)
  =\ & d_{f}\,d^{*}_{f}\,v+z\,\big(\Delta_{f}^{\mathbf{b},(p)}-z\big)^{-1}d_{f}\,d^{*}_{f}\,v\\
 \label{eq.commut-z'}
  \underset{z\to 0^{-}}{\longrightarrow}&\  d_{f}\big(d^{*}_{f}\,v\big)
\end{align}
and we deduce from \eqref{eq.commut-z} and \eqref{eq.commut-z'} that
\begin{align*}
d_{f}\,d^{*}_{f}\,v\,\in\,D(\Delta_{f}^{\mathbf{b},(p)})
\quad\text{and}\quad 
\Delta_{f}^{\mathbf{b},(p)}\,d_{f}\,d^{*}_{f}\,v\ &=\ 
d_{f}\,d^{*}_{f}(u-\pi_{f,\mathbf b}u)\\
\ &=\ 
d_{f}\big(d^{*}_{f}\,u\big)\,,
\end{align*}
which proves the first equality of \eqref{eq.ortho-decomp'} according to 
\eqref{eq.proj-Delta-u} and \eqref{eq.proj-u}. The second equality of 
\eqref{eq.ortho-decomp'} is proven similarly after establishing the analogous
versions of \eqref{eq.commut-z} and \eqref{eq.commut-z'}
with
$d_{f}\,d^{*}_{f}$ replaced by $d^{*}_{f}\,d_{f}$.
The last part of Proposition~\ref{pr.ortho-decomp}
then  follows easily, using again iv) of Proposition~\ref{pr.Witten}
as in \eqref{eq.commut-z} and \eqref{eq.commut-z'}
to obtain \eqref{eq.ortho-decomp''} and \eqref{eq.ortho-decomp'''}.
\end{proof}

\section{Proofs of the main results}
\label{se.proofs}

\subsection{Proof of Theorem~\ref{th.Witten-Bochner}}

We first prove Theorem~\ref{th.Witten-Bochner} in the   case  $f=0$. As shown
 in \cite{Sch}, it implies
 in particular Gaffney's inequalities which state the equivalence between
 the norms
 $\|\cdot\|_{\Lambda^{p}H^{1}}$
 and $\sqrt{\mathcal D^{(p)}(\cdot)+\|\cdot\|^{2}_{\Lambda^{p}L^{2}}}$
 for tangential or normal $p$-forms.  
  
 \begin{theorem}\label{th.H1} Let $\omega\in \Lambda^{p}H_{\mathbf b}^{1}$
  with $\mathbf b \in \{\mathbf n,\mathbf t\}$.
We have then the identity
 $$
 \|\omega\|^{2}_{\dot H^{1}}
 \ =\ 
 \mathcal D^{(p)}(\omega) -\langle {\rm{Ric}}^{(p)}\omega ,\omega \rangle_{L^{2}}
 - \int_{\partial\Omega} \langle \mathcal K_{\mathbf b}^{(p)}\omega ,\omega \rangle_{\Lambda^{p}}~d\mu_{\partial\Omega}\,,
$$
where
${\rm{Ric}}^{(p)} \in \mathcal L(\Lambda^{p}T^{*}\Omega)$ and
$
\mathcal K_{\mathbf b}^{(p)}\in \mathcal L(\Lambda^{p}T^{*}\Omega\big|_{\partial\Omega})
$  
 have  been respectively defined in
 \eqref{eq.Ric-p}
and  in \eqref{eq.Kn1}--\eqref{eq.Ktp}, and 
$$\| \cdot\|^{2}_{\dot H^{1}}\ :=\
\| \cdot\|^{2}_{H^{1}}-\| \cdot\|^{2}_{L^{2}}\,.$$
\end{theorem}

\noindent
The statement of Theorem~\ref{th.H1} is essentially the statement of \cite[Theorem~2.1.5]{Sch}
and the content of its proof in the case $\mathbf t \omega =0$, and
is closely related to the statement of \cite[Theorem~2.1.7]{Sch}
in the case $\mathbf n \omega =0$. We have nevertheless
 to compute the exact form of $\mathcal K_{\mathbf n}^{(p)}$ and especially
of $\mathcal K_{\mathbf n}^{(1)}$ in the latter case. We also give a complete proof in the case 
$\mathbf t \omega =0$ for the sake of clarity.

\begin{proof}
By density of $\Lambda^{p}\mathcal{C}^{\infty}_{\mathbf{b}}$
in $\Lambda^{p}H^{1}_{\mathbf b}$ for $\mathbf b\in \{\mathbf t,\mathbf n\}$,
it is sufficient to prove Theorem~\ref{th.H1} for $\omega\in \Lambda^{p}\mathcal C_{\mathbf{b}}^{\infty}$.
Moreover,
it follows from the Weitzenb\"ock formula \eqref{eq.Weitzen} and from the Green's formulas for the Hodge
and Bochner
Laplacians,
\begin{multline}
\label{Green-Laplacian} 
 \mathcal D^{(p)}(\omega)
=
\langle \Delta\omega,\omega\rangle_{\Lambda^{p}L^{2}}
+\int_{\partial\Omega}\big(\langle\textbf{i}_{\vec n}d\omega, \omega\rangle_{\Lambda^{p}}
-
\langle d^{*}\omega,\textbf{i}_{\vec n}\omega\rangle_{\Lambda^{p-1}} \big) \, d\mu_{\partial\Omega}
\end{multline}
and
\begin{align}
\label{H1-Bochner} 
 \|\omega\|^{2}_{\Lambda^{p}\dot H^{1}}
\ =\ 
 \langle\Delta_{B}\omega,\omega\rangle_{\Lambda^{p}L^{2}}
+\int_{\partial\Omega}\langle\nabla_{\vec n} \omega,\omega\rangle_{\Lambda^{p}} d\mu_{\partial\Omega}\,,
\end{align}
that
for any $\omega\in\Lambda^{p}\mathcal{C}^{\infty}$, the
 expression
 $$
\|\omega\|^{2}_{\Lambda^{p}\dot H^{1}}-
 \mathcal D^{(p)}(\omega) +\langle {\rm{Ric}}^{(p)}\omega ,\omega \rangle_{\Lambda^{p}L^{2}}
$$
 reduces to the boundary integral 
$$
\int_{\partial\Omega}\big(\langle\nabla_{\vec n} \omega,\omega\rangle_{\Lambda^{p}}-\langle\textbf{i}_{\vec n}d\omega, \omega\rangle_{\Lambda^{p}}
+
\langle d^{*}\omega,\textbf{i}_{\vec n}\omega\rangle_{\Lambda^{p-1}} \big) \, d\mu_{\partial\Omega}
$$
and we have then just to check that for any $\omega$ in $\Lambda^{p}\mathcal{C}^{\infty}_{\mathbf{b}}$, it holds
 \begin{equation}
 \label{eq.bdy-term}
 \langle \mathcal K_{\mathbf b}^{(p)}\omega,\omega\rangle_{\Lambda^{p}}\ =\ 
 \langle \textbf{i}_{\vec n}d\omega,\omega\rangle_{\Lambda^{p}} 
 -\langle d^{*}\omega,\textbf{i}_{\vec n}\omega\rangle_{\Lambda^{p-1}}
-\langle\nabla_{\vec n} \omega,\omega\rangle_{\Lambda^{p}}
\,,
\end{equation}
where  $\mathcal K_{\mathbf b}^{(p)}\in \mathcal L(\Lambda^{p}T^{*}\Omega\big|_{\partial\Omega})$
has been defined in \eqref{eq.Kn1}--\eqref{eq.Ktp}.\\
\\
\textbf{Case $\mathbf n \omega=0$ :}\\[0.1cm] 
 We have then $\langle d^{*}\omega,\textbf{i}_{\vec n}\omega\rangle_{\Lambda^{p}}=0$
and
 $$
 \langle \textbf{i}_{\vec n}d\omega,\omega\rangle_{\Lambda^{p}} - \langle\nabla_{\vec n} \omega,\omega\rangle_{\Lambda^{p}}
 \ =\ \langle \textbf{i}_{\vec n}d\omega-\nabla_{\vec n}\omega,\omega\rangle_{\Lambda^{p}} 
 \ =\ \langle \mathbf t\big(\textbf{i}_{\vec n}d\omega-\nabla_{\vec n}\omega\big),\omega\rangle_{\Lambda^{p}}
\,,$$
the last equality following again from $\mathbf n \omega=0$. It is then sufficient to show that
for any $\omega$ in $\Lambda^{p}\mathcal{C}^{\infty}_{\mathbf{n}}$, it holds
\begin{equation}
\label{eq.Kn-n}
\mathcal K_{\mathbf n}^{(p)}\omega\ =\ \mathbf t\big(\textbf{i}_{\vec n}d\omega-\nabla_{\vec n}\omega\big)\,.
\end{equation}

\noindent
Taking now $p$ tangential vector fields $X_{1},\dots, X_{p}$
and denoting for simplicity $\vec n$ by $X_{0}$,
we deduce from \eqref{eq-nabla-d} that: 
\begin{align*}
(\textbf{i}_{X_{0}}d\omega-\nabla_{X_{0}}\omega)(X_{1},\dots,X_{p})&=d\omega(X_{0},X_{1},\dots,X_{p})
-(\nabla_{X_{0}}\omega)(X_{1},\dots,X_{p})
\\
&=\sum_{k=1}^{p}(-1)^{k}(\nabla_{X_{k}}\omega)(X_{0},\dots,\dot{X_{k}},\dots, X_{p})\,.
\end{align*}
Moreover, using \eqref{eq.nabla-form}, the tangentiality of $X_{1},\dots, X_{p}$,  and $\mathbf n \omega=0$, we have for any $k\in\{1,\dots,p\}$:
 \begin{align*}
(\nabla_{X_{k}}\omega)(X_{0},\dots,\dot{X_{k}},\dots, X_{p})
\ =\ &\nabla_{X_{k}}\big(\omega(X_{0},\dots,\dot{X_{k}},\dots, X_{p})\big)\\
-\sum_{k\neq \ell=0,\dots,p}&\omega(X_{0},\dots,\nabla_{X_{k}}X_{\ell},\dots,\dot{X_{k}},\dots,X_{p})\\
\ =\ &
-\omega(\nabla_{X_{k}}X_{0},\dots,\dot{X_{k}},\dots,X_{p})\\
\ =\ &(-1)^{k}\omega(X_{1},\dots,\nabla_{X_{k}}X_{0},\dots,X_{p})\,.
\end{align*}
Hence, it holds for any $\omega\in \Lambda^{p}\mathcal{C}^{\infty}_{\mathbf{n}}$ and $p$ tangential vector fields $X_{1},\dots, X_{p}$, 
\begin{align*}
(\textbf{i}_{\vec n}d\omega-\nabla_{\vec n}\omega)(X_{1},\dots,X_{p})\ =\ 
\sum_{k=1}^{p}\omega(X_{1},\dots,\nabla_{X_{k}}\vec n,\dots,X_{p})\,,
\end{align*}
the r.h.s. being nothing but 
$\big(\,\mathcal K_{\mathbf n}^{(p)}\omega\,\big)(X_{1},\dots,X_{p})$
according to \eqref{eq.Kn1} and \eqref{eq.Knp}. This proves
\eqref{eq.Kn-n} and then concludes the proof in the case $\mathbf n\omega=0$.

\noindent\\
\textbf{Case $\mathbf t \omega=0$ :}\\[0.1cm] 
 We have here  $\langle \textbf{i}_{\vec n}d\omega,\omega\rangle_{\Lambda^{p}}=0$
 and from \eqref{eq.bdy-term}, we are then
 led  to compute more precisely
\begin{align*}
\langle d^{*}\omega,\textbf{i}_{\vec n}\omega\rangle_{\Lambda^{p-1}} + \langle\nabla_{\vec n} \omega,\omega\rangle_{\Lambda^{p}}
 &\ =\ \langle \vec n^{\flat}\wedge d^{*}\omega+ \nabla_{\vec n}\omega,\omega\rangle_{\Lambda^{p}}\\
 &\ =\ \langle \mathbf n\big( \vec n^{\flat}\wedge d^{*}\omega+ \nabla_{\vec n} \omega\big),\omega\rangle_{\Lambda^{p}}\\
& \ =\ 
\langle\vec n^{\flat}\wedge( d^{*}\omega+ \mathbf i_{\vec n}\nabla_{\vec n} \omega),\omega\rangle_{\Lambda^{p}}
\,,
\end{align*}
the second to last equality following from $\mathbf t \omega=0$
and the  last one from $\mathbf n \omega=\vec n^{\flat}\wedge (\mathbf i_{\vec n}\omega)$.
To conclude, it then remains to show that
for any $\omega$ in $\Lambda^{p}\mathcal{C}^{\infty}_{\mathbf{t}}$, it holds
\begin{equation}
\label{eq.K-t}
\mathcal K_{\mathbf t}^{(p)}\omega\ =\ -\mathbf n\big( \vec n^{\flat}\wedge d^{*}\omega+ \nabla_{\vec n} \omega\big)\ =\ -\vec n^{\flat}\wedge( d^{*}\omega+ \mathbf i_{\vec n}\nabla_{\vec n} \omega)\,.
\end{equation}
Denoting  by $(E_{1},\dots,E_{n})$ a local orthonormal frame such that $E_{n}=\vec n$ on $\partial\Omega$
and using \eqref{eq.d*}, we get
$$
-\mathbf n\big( \vec n^{\flat}\wedge d^{*}\omega+ \nabla_{\vec n} \omega\big)
\ =\ 
\vec n^{\flat}\wedge\big(\,\sum_{i=1}^{n-1}\mathbf i_{E_{i}}\nabla_{E_{i}}\omega\,\big)\,.
$$
Taking now $p-1$ tangential vector fields $X_{1},\dots,X_{p-1}$,
we then have:
\begin{align*}
-\mathbf n\big( \vec n^{\flat}\wedge d^{*}\omega+ \nabla_{\vec n} \omega\big)(\vec n,X_{1},\dots,X_{p-1})
\ =\    \sum_{i=1}^{n-1}(\nabla_{E_{i}}\omega)(E_{i},X_{1},\dots,X_{p-1})\,,
\end{align*}
where, for any $i\in\{1,\dots,n-1\}$, using \eqref{eq.nabla-form},
the tangentiality of the vector fields $X_{1},\dots,X_{p-1}$,
 $\mathbf t\omega=0$,
and denoting for simplicity $E_{i}$ by $X_{0}$,
\begin{align*}
(\nabla_{X_{0}}\omega)(X_{0},X_{1},\dots,X_{p-1})
&\  =\   \nabla_{X_{0}}\big(\omega(X_{0},X_{1},\dots,X_{p-1})\big)\\
&\qquad-\sum_{\ell=0}^{p-1}\omega(X_{0},\dots,\nabla_{X_{0}}X_{\ell},\dots,X_{p-1})\\
&\ =\ -\sum_{\ell=0}^{p-1}\omega(X_{0},\dots,(\nabla_{X_{0}}X_{\ell})^{\perp},\dots,X_{p-1})\\
&\ =\ -\sum_{\ell=0}^{p-1}\omega(X_{0},\dots,\mathcal K_{2}(X_{0},X_{\ell}),\dots,X_{p-1})\\
&\ =\ -\Big(\big(\mathcal K_{2}(X_{0},\cdot)\big)^{(p)}\omega\Big)(X_{0},\dots,X_{p-1})\,,
\end{align*}
where the notation $\big(\mathcal K_{2}(X_{0},\cdot)\big)^{(p)}\omega$
has been defined at the line following \eqref{eq.Ktp}.
Consequently, it holds for any $\omega\in \Lambda^{p}\mathcal{C}^{\infty}_{\mathbf{t}}$ and $p-1$ tangential vector fields $X_{1},\dots, X_{p-1}$, 
\begin{align*}
-\mathbf n\big( \vec n^{\flat}\wedge d^{*}\omega+ \nabla_{\vec n} \omega\big)(&\vec n,X_{1},\dots,X_{p-1})\\
\ &=\ -\sum_{i=1}^{n-1}\Big(\big(\mathcal K_{2}(E_{i},\cdot)\big)^{(p)}\omega\Big)(E_{i},X_{1},\dots,X_{p-1})\\
\ &=\ \big(\mathcal K_{\mathbf t}^{(p)}\omega\big)(\vec n,X_{1},\dots,X_{p-1})\,,
\end{align*}
which proves \eqref{eq.K-t} and concludes the proof of Theorem~\ref{th.H1}.
\end{proof}
\noindent
We end up this subsection with the 
proof of Theorem~\ref{th.Witten-Bochner}.
\begin{proof}[Proof of Theorem~\ref{th.Witten-Bochner}]
According to \eqref{eq.Hess}, \eqref{eq.df4}, Lemma~\ref{le.intbypart}, and to Theorem~\ref{th.H1},
we have just to show the identity
\begin{multline}
\| e^{f }\omega\|^{2}_{\Lambda^{p}\dot H^{1}(e^{-2f}d\mu)}
=
\|\omega\|^{2}_{\Lambda^{p}\dot H^{1}}
+\langle(\left|\nabla f\right|^{2}+\Delta f)\omega,\omega\rangle_{\Lambda^{p} L^{2}}\\
\label{eq.id-to-show}
+\int_{\partial \Omega}\langle\omega,\omega
\rangle_{\Lambda^{p}}\,\partial_{n}f~d \mu_{\partial \Omega}\,.
\end{multline}
Let now $(U_{j})_{j\in\{1,\dots,K\}}$ be any nice cover of $\Omega$ with associated local orthonormal frames
 $(E_{1},\dots,E_{n})$ (we drop the dependence on $j$ to lighten the notation)  and a subordinated partition of unity
 $(\rho_{j})_{j\in\{1,\dots,K\}}$. We have then the relation (see
\cite[p.~31]{Sch} for more details about the $H^{1}$ norm):
 \begin{equation*}
 \label{eq.H1-L2-0}
\| e^{f }\omega\|^{2}_{\Lambda^{p}\dot H^{1}(e^{-2f}d\mu)}\  = \ \sum_{j=1}^{K} \sum_{i=1}^{n}\int_{U_{j}} \rho_{j} \,\big\| e^{- f }\nabla_{E_{i}}(e^{f }\omega)\big\|^{2}_{\Lambda^{p}}\,d\mu\,.
 \end{equation*}
Moreover,  the relation $e^{- f }\nabla_{E_{i}}(\,e^{f }\cdot\,)=(\nabla_{E_{i}}f)\cdot\,+\,\nabla_{E_{i}}(\,\cdot\,)$ implies
\begin{align*}
\big\| e^{- f }\nabla_{E_{i}}(e^{f }\omega)\big\|^{2}_{\Lambda^{p}}= 
\big\| \nabla_{E_{i}}\omega\big\|^{2}_{\Lambda^{p}}+\big\|(\nabla_{E_{i}}f)\omega\big\|^{2}_{\Lambda^{p}}\!+2\langle \nabla_{E_{i}}\omega,(\nabla_{E_{i}}f)\omega\rangle_{\Lambda^{p}}
\end{align*}
so according to \eqref{eq.id-to-show}, we are simply led  to prove that
$$
2\sum_{j=1}^{K} \sum_{i=1}^{n}\rho_{j} \int_{U_{j}}\langle \nabla_{E_{i}}\omega,(\nabla_{E_{i}}f)\omega\rangle_{\Lambda^{p}}d\mu=
\langle(\Delta f)\omega,\omega\rangle_{L^{2}}
+\int_{\partial \Omega}\|\omega\|^{2}_{\Lambda^{p}}\,\partial_{n}f~d \mu_{\partial \Omega}\,.
$$
To conclude, we use the Green's formula \eqref{usual-Green'}
and  the formula \eqref{eq.d*}
 for the codifferential which give
\begin{align*}
\int_{\partial \Omega}\|\omega\|^{2}_{\Lambda^{p}}\,\partial_{n}f~d \mu_{\partial \Omega}
&= -\int_{\Omega} d^{*}\big(\|\omega\|^{2}_{\Lambda^{p}}\, df\big)d\mu\\
&= \sum_{j=1}^{K}\sum_{i=1}^{n}\int_{U_{j}}\rho_{j}\,\mathbf i_{E_{i}}\nabla_{E_{i}}\big(\|\omega\|^{2}_{\Lambda^{p}}\, df\big)d\mu\\
&=\sum_{j=1}^{K}\sum_{i=1}^{n}\int_{U_{j}}\rho_{j}\Big(2\langle\nabla_{E_{i}}\omega,\omega \rangle_{\Lambda^{p}} df(E_{i}) \\
&\qquad\qquad\qquad\qquad\qquad\qquad+ \|\omega\|^{2}_{\Lambda^{p}} (\nabla_{E_{i}}df)(E_{i})\Big)d\mu\\
&=
2\sum_{j=1}^{K}\sum_{i=1}^{n}\int_{U_{j}}\rho_{j}\langle \nabla_{E_{i}}\omega,(\nabla_{E_{i}}f)\omega\rangle d\mu
-
\langle(\Delta f)\omega,\omega\rangle_{L^{2}}.
\end{align*}
This implies \eqref{eq.id-to-show} and then concludes the proof of Theorem~\ref{th.Witten-Bochner}.
\end{proof}

\subsection{Proof of Theorem~\ref{th.main}}

We first prove $1.i)$ and then consider $p>0$ and $\omega\in \Lambda^{p-1}H_{\mathbf n}^{1}(d\nu)$
such that $d^{*}_{V}\omega =0$. Let us also consider the corresponding
form on the flat space: $$
\eta\ :=\ e^{-f}\,\omega\quad\text{where}\quad  f\ :=\ \frac V2\,.
$$
We have then in particular $\eta\,\in\, \Lambda^{p-1}H_{\mathbf n}^{1}$
and $d^{*}_{f}\eta =0$.
Note also
 that $\mathcal K_{\textbf n}^{(p)}\geq 0$ and ${\rm{Ric}}^{(p)}+2\,\Hess^{\!(p)}\!f>0$
together with Theorem~\ref{th.Witten-Bochner} imply that 
\begin{equation}
\label{eq.lowerbound}
\Delta_{f}^{\mathbf{n},(p)}\ \geq\ 
{\rm{Ric}}^{(p)}+2\,\Hess^{\!(p)}\!f\ >\ 0
\end{equation}
and therefore that
$0\in \varrho(\Delta_{f}^{\mathbf{n},(p)})$.
As already explained in the introduction (see \eqref{eq.HS0} there),
the trick   is to use now the following relation
which results easily from $d^{*}_{f}\eta=0$, \eqref{Green-distorted}, \eqref{eq.ortho-decomp''}, and \eqref{eq.ortho-decomp'''}:
\begin{align}
\nonumber
\left\| \,\eta - \pi_{f,\mathbf n}\eta\, \right\|_{L^{2}}^{2}
\ &=\ \langle \,(\Delta_{f}^{\mathbf{n},(p)})^{-1}\, d_{f}(\eta- \pi_{f,\mathbf n}\eta)\,,\,
d_{f}(\eta - \pi_{f,\mathbf n}\eta)\, \rangle_{L^{2}} \\
\ &=\ \langle \,(\Delta_{f}^{\mathbf{n},(p)})^{-1}\, d_{f}\eta\,,\,
d_{f}\eta\, \rangle_{L^{2}}\,,
\label{eq.HS}
\end{align}
where
$\pi_{f,\mathbf n}=\pi^{(p)}_{f,\mathbf n}$ denotes 
the orthogonal projection on $\Ker(\Delta^{\mathbf n,(p)}_{f})$.
The estimate to prove involving $\omega=e^{f}\,\eta$  is then a simple consequence of
 \eqref{eq.lowerbound}  and \eqref{eq.HS} 
 according to the  unitary equivalence
 $$
 L^{\mathbf n,(p)}_{V}\ =\ e^{f}\,\Delta^{\mathbf n,(p)}_{f}\,e^{-f}
 \quad\text{where}\quad  f\ =\ \frac V2
 \,.
 $$
The proof of 1.ii) is completely similar as well as the proofs 
of 2.i) and 2.ii).

\subsection{Proof of Corollary~\ref{co.main}}

This proof is similar to the previous one and we only prove it in the normal case, the tangential
case being completely analogous. In order to improve the latter result, 
we want to derive an estimate of the type  
\eqref{eq.lowerbound}
with $\Delta_{f}^{\mathbf{n},(1)}$ replaced by the self-adjoint unbounded operator $\Delta_{f}^{\mathbf{n},(1)}\big|_{\Ran\,d_{f}}$
defined in
 Proposition~\ref{pr.ortho-decomp}. To do so, we will in particular 
 make use of the nonnegative term
$\| e^{f}\cdot\|^{2}_{\dot H^{1}(e^{-2f}d\mu)}$
of the
integration by part formula \eqref{eq.Witten-Bochner}
stated  in Theorem~\ref{th.Witten-Bochner}.\\[0.1cm]
Let us then consider $\omega\in D(L_{V}^{\mathbf n,(0)})=\{ u\in H^{2}\cap H_{\mathbf n}^{1}(d\nu)\ \text{s.t.}\ \mathbf n du=0\ \text{on}\ \partial\Omega\}$
and
its corresponding
form on the flat space $$
\eta\ :=\ e^{-f}\omega\quad\text{where}\quad  f\ :=\ \frac V2\
$$
which consequently belongs to $D(\Delta^{\mathbf n,(0)}_{f})$.  Denoting by $(E_{1},\dots,E_{n})$  a local orthonormal frame
on $U\subset \Omega$,
we deduce
from the  Cauchy-Schwarz inequality the following relations
satisfied by the integrand of $\| e^{f}d_{f}\eta\|^{2}_{\dot H^{1}(e^{-2f}d\mu)}$ a.e. on $U$
and for every $N$ such that $\frac{1}{N}\in [-\infty,\frac1n)$ or $N=n$ if $f$ is constant:
\begin{align*}
\sum_{i=1}^{n}\big\| e^{- f }\nabla_{E_{i}}(e^{f } \,d_{f}\eta)\big\|^{2}_{\Lambda^{1}}
\ &\geq\ \frac{1}{n}\big( e^{- f }\,\Delta^{(0)}\,e^{f }\eta\big)^{2}\\
\ &=\ \frac{1}{n}\big( \Delta^{(0)}_{f}\eta-2\,\langle df,d_{f}\eta\rangle_{\Lambda^{1}}\big)^{2}\\
\ &\geq\  \frac{1}{N}\big( \Delta^{(0)}_{f}\eta\big)^{2}-\frac{4}{N-n}\,df\otimes df(d_{f}\eta,d_{f}\eta)\,.
\end{align*}
This implies, after integration on $\Omega$:
 \begin{align}
 \nonumber
\| e^{f}d_{f}\eta\|^{2}_{\dot H^{1}(e^{-2f}d\mu)}\,+\,
&\frac{4}{N-n}\int_{\Omega}df\otimes df(d_{f}\eta,d_{f}\eta)~d\mu\\
\label{eq.gainN}
\ &\geq\  \frac{1}{N}\| \Delta^{(0)}_{f}\eta\|_{L^{2}}^{2}\ =\ \frac{1}{N}\mathcal D^{(1)}_{f}(d_{f}\eta)
\,.
\end{align}
Moreover, $\mathcal K_{\textbf n}^{(1)}\geq 0$ and
$$
{\rm{Ric}}_{2f,N}\ :=\ {\rm Ric}\,+\,2\,\Hess f\,-\,\frac{4}{N-n}\,df\otimes df\ >\ 0
$$
together with Theorem~\ref{th.Witten-Bochner} and \eqref{eq.gainN} imply that 
\begin{equation}
\label{eq.lowerbound-BE-tensor}
(1-\frac1N)\Delta_{f}^{\mathbf{n},(1)}\big|_{\Ran\,d_{f}}\ \geq\ 
{\rm{Ric}}_{2f,N}\ >\ 0\,.
\end{equation}
The estimate to prove is then a simple consequence of
 \eqref{eq.lowerbound-BE-tensor} and of the
 relation   \begin{equation}
 \label{eq.HS'}
 \left\| \,\eta - \pi_{f,\mathbf n}\eta\, \right\|_{L^{2}}^{2}
 = \langle \,(\Delta_{f}^{\mathbf{n},(1)}\big|_{\Ran\,d_{f}})^{-1}\, d_{f}(\eta- \pi_{f,\mathbf n}\eta)\,,\,
d_{f}(\eta - \pi_{f,\mathbf n}\eta)\, \rangle_{L^{2}}
\end{equation}
 valid for any $\eta\in \Lambda^{0}H^{1}_{\mathbf n}=H^{1}(\Omega)$
 and resulting from \eqref{Green-distorted}, \eqref{eq.ortho-decomp''}, and \eqref{eq.ortho-decomp'''}.

\

\noindent
{\bf Acknowledgements}\\
 \noindent
This work was completed during a ``D\'el\'egation INRIA'' at CERMICS.
The author  thanks people there for their hospitality.



\begin{thebibliography} {99}
  \bibitem{BaJeSj} Bach, V., Jecko, T., Sj\"ostrand, J.:
 Correlation asymptotics of classical lattice spin systems with nonconvex Hamilton function at low temperature.
Ann. Henri Poincar\'e 1, no. 1, 59--100 (2000). 

 \bibitem{BaMo} Bach, V., M{\o}ller, J.S.: Correlation at low temperature: I. Exponential
decay. J. Funct. Anal. 203, no. 1, 93--148 (2003).
 
 \bibitem{BaMo2} Bach, V., M{\o}ller, J.S.: Correlation at low temperature: II. Asymptotics. 
 J. Statist. Phys. 116, no. 1--4, 591--628 (2004).
 
 \bibitem{BaEm} Bakry, D., \'{E}mery, M.:
 \newblock Diffusions hypercontractives. 
  \newblock  Sem. Probab. XIX, Lecture Notes in Math.
  1123, 177--206, Springer  (1985).
  
  
\bibitem{BaGeLe} Bakry, D., Gentil, I., Ledoux, M.:
\newblock\textit{Analysis and geometry of Markov diffusion operators.}
\newblock Grund. der Math. Wiss. 348, Springer~(2014).

\bibitem{Boc} Bochner, S.: Curvature and Betti numbers. Ann. of Math. 49 (2), 379--390 (1948).

\bibitem{BrLi} Brascamp, H.J., Lieb, E.H.:
On extensions of the Brunn-Minkowski and
Pr\'ekopa-Leindler theorems, including inequalities for 
log concave functions, and
with an application to the diffusion equation. J. Funct. Anal. 22 (4), 366--389 (1976).

\bibitem{DiLe} Di~Ges\`u, G., Le~Peutrec, D.:  
Small noise spectral gap asymptotics for a large system of nonlinear diffusions.
To appear in J. Spectr. Theory,
preprint on Arxiv, \url{http://arxiv.org/abs/1506.04434}, 47 pages (2015).

\bibitem{GaMe}
Gallot, S., Meyer, D.:
Op\'erateur de courbure et laplacien des formes diff\'erentielles d'une vari\'et\'e riemannienne.
J. Math. Pures Appl.  54 (9), no. 3, 259--284 (1975).
  


\bibitem{Hel} Helffer, B.: Remarks on the decay of correlations and Witten Laplacians -- the Brascamp-Lieb inequality and
semiclassical limit. J. Funct. Anal. 155, 571--586
 (1998).

 

 
 \bibitem{HelNi1} Helffer, B., Nier, F.:
 \newblock \textit{Quantitative analysis of metastability in 
reversible diffusion processes via a Witten complex approach: the case with boundary}.
 \newblock M\'emoire 105, Soci\'et\'e Math\'ematique de France (2006).


\bibitem{HeSj} Helffer, B., Sj\"ostrand, J.: On the correlations for Kac like models in the
convex case. J. Stat. Phys. 74, 349--409 (1994).

\bibitem{Jam} Jammes, P.:
Sur la multiplicit\'e des valeurs propres du laplacien de Witten.
Trans. Amer. Math. Soc. 364 (6), 2825--2845 (2012).

\bibitem{Joh} Johnsen, J.: On the spectral properties of Witten Laplacians, their
range projections and Brascamp-Lieb's inequality. Integr. Equ. Oper. Theory 36(3), 288--324 (2000).

\bibitem{KoMi} Kolesnikov, A.V.,  Milman, E.:
Brascamp-Lieb-Type Inequalities on Weighted Riemannian Manifolds with Boundary.
J. Geom. Anal., doi:10.1007/s12220-016-9736-5, 23 pages (2016).

\bibitem{KoMi2} Kolesnikov, A.V.,  Milman, E.:
Riemannian metrics on convex sets with applications to
Poincar\'e and log-Sobolev inequalities. Calc. Var.
55: 77, doi:10.1007/s00526-016-1018-3, 36 pages  (2016).



\bibitem{Lep}
Le Peutrec, D.:
\newblock Small eigenvalues of the Neumann realization
 of the semiclassical Witten Laplacian.
\newblock Ann. de la Facult\'{e} des Sciences de 
Toulouse,  Vol.~19, no. 3--4, 735--809 (2010).


\bibitem{Lic} Lichnerowicz, A.: Variétés riemanniennes \`a tenseur C non négatif.
C. R. Acad. Sci. Paris S\'er. A-B 271, A650--A653 (1970). 

\bibitem{Lot} Lott, J.:
Some geometric properties of the Bakry-\'Emery-Ricci tensor.
Comment. Math. Helv. 78, no. 4, 865--883 (2003).

\bibitem{LoVi} Lott, J., Villani, C.:
 Ricci curvature for metric-measure spaces via optimal
transport.
Ann. of Math. (2) 169, no. 3, 903--991 (2009).

\bibitem{NaSp} Naddaf, A., Spencer, T.:
On homogenization and scaling limit of some gradient perturbations of a massless free field.
Comm. Math. Phys. 183, no. 1, 55--84  (1997).




\bibitem{Sjo} Sj\"ostrand, J.: Correlation asymptotics and Witten Laplacians.  Algebra i Analiz 8, no. 1, 160--191 (1996).  

 \bibitem{Sch} Schwarz, G.:
 \newblock{\it Hodge decomposition. A method for Solving Boundary Value
 Problems.}
 \newblock Lecture Notes in Mathematics 1607, Springer Verlag (1995).
 
 \bibitem{Wit} Witten, E.:
 \newblock Supersymmetry and Morse inequalities.
 \newblock J.~Diff.~Geom.~17, 661--692 (1982).
 


 

 \end{thebibliography}
\end{document}